\documentclass[11pt]{amsart}
\usepackage{amsmath,amssymb,amsthm,amsfonts,setspace,hyperref,dsfont}

\newcommand{\NN}{\mathbb{N}}
\newcommand{\RR}{\mathbb{R}}

\newcommand{\CC}{\mathbb{C}}

\newcommand{\TT}{\mathbb{T}}
\newcommand{\ZZ}{\mathbb{Z}}
\newcommand{\GG}{\mathbb{G}}

\newcommand{\norm}[1]{\lVert#1\rVert}
\newcommand{\abs}[1]{\lvert#1\rvert}

\newtheorem{theorem}{Theorem}[section]
\newtheorem{corollary}[theorem]{Corollary}

\newtheorem{lemma}[theorem]{Lemma}

\theoremstyle{definition}
\newtheorem{definition}[theorem]{Definition}
\newtheorem{remark}[theorem]{Remark}
\newtheorem{fact}[theorem]{Fact}

\numberwithin{equation}{section}

\DeclareMathOperator{\Ad}{Ad}

\begin{document}
\title[Cocycle rigidity of abelian partially hyperbolic actions]
{Cocycle rigidity of abelian partially hyperbolic actions}
\author[]{ Zhenqi Jenny Wang$^1$ }
\address{Department of Mathematics\\
        Michigan State University\\
        East Lansing, MI 48824,USA}
\email{wangzq@math.msu.edu}

\thanks{ $^1)$ Based on research  supported by NSF grant DMS-1346876}

\begin{abstract} Suppose $G$ is a higher-rank connected  semisimple Lie group with finite center and without compact factors. Let $\mathbb{G}=G$ or $\mathbb{G}=G\ltimes V$, where $V$ is a finite dimensional vector space $V$.  For any unitary representation $(\pi,\mathcal{H})$ of $\GG$, we study the twisted cohomological
equation $\pi(a)f-\lambda f=g$ for partially hyperbolic element $a\in \mathbb{G}$ and $\lambda\in U(1)$, as well as the twisted cocycle equation $\pi(a_1)f-\lambda_1f=\pi(a_2)g-\lambda_2 g$ for commuting partially hyperbolic elements $a_1,\,a_2\in \mathbb{G}$. We characterize the obstructions to solving these equations, construct
smooth solutions and obtain tame Sobolev estimates for the solutions. These results can be extended to partially hyperbolic flows parallelly.

As an application, we prove cocycle rigidity  for any abelian higher-rank partially hyperbolic algebraic actions. This is the first paper exploring rigidity properties of partially hyperbolic that the hyperbolic directions
don't generate the whole tangent space. The result can be viewed as a first step toward the application of KAM method in obtaining differential rigidity for these actions  in future works.
\end{abstract}


\maketitle
\section{Introduction}
\subsection{Various abelian algebraic actions} \label{sec:3} We define $\ZZ^k\times \RR^\ell$, $k+\ell\geq 1$ algebraic actions as follows. Let $H$ be a
connected Lie group, $A\subseteq H$ a closed abelian subgroup which
is isomorphic to $\ZZ^k\times \RR^\ell$, $L$ a compact subgroup of
the centralizer $Z(A)$ of $A$, and $\Gamma$ a torsion free lattice in
$H$. Then $A$ acts by left translation on the compact space
$\mathcal{M}=L\backslash H/\Gamma$. Denote this action by $\alpha_A$. The three specific
types of examples discussed below correspond to:
\begin{itemize}
  \item  for the symmetric space examples take $H$ a semisimple Lie group of the non-compact
type.

\smallskip
  \item  for the twisted symmetric space examples take $H=G\ltimes _\rho\RR^m$ or $H=G\ltimes_\rho N$, a semidirect
product of a reductive Lie group $G$ with semisimple factor of the non-compact
type with $\RR^m$ or a simply connected nilpotent group $N$.

\smallskip
  \item for the parabolic action examples, take $H$ a semisimple Lie group of the non-compact
type and $A$ a subgroup of a maximal abelian unipotent subgroup in $H$.
\end{itemize}

\subsection{History and method}In contrast to the classical rank-one actions where Livsic showed that
there is an infinite-dimensional space of obstructions to solving the cohomological
equation for a hyperbolic action by $\RR$ or $\ZZ$, in the past two decades various rigidity phenomena for (partially) hyperbolic actions have been well understood. Significant progresses have been made in the case of cocycle rigidity for higher rank (partially) hyperbolic algebraic actions (see \cite{Damjanovic1}, \cite{Kononenko}, \cite{Spatzier1}, \cite{Spatzier2} and \cite{Zhenqi0}) obtained from symmetric and twisted symmetric space examples. In these papers, the higher rank property is used to show the existence of a distributional or continuous transfer function; then the smoothness of the transfer function
follows from the fact that it is smooth along stable and unstable directions
and that those generate the tangent space at every point. Hence all actions considered in previous papers satisfy the following property which is essential for obtaining smooth rigidity:

\medskip
\noindent ($\mathfrak B$) {\em The stable directions of various  action elements   generate  the tangent space as a Lie algebra.}
\medskip

In \cite{Spatzier1} and \cite{Spatzier2} the proofs are based on harmonic analysis of
semisimple Lie groups, specifically, on exponential decay of matrix coefficients of partially hyperbolic elements. In  \cite{Damjanovic1}, \cite{Kononenko} and \cite{Zhenqi0} the main geometric ingredient is the accessibility of stable and unstable foliations, which enables the construction of continuous transfer function. The natural difficulty in extending the rigidity results to general partially hyperbolic algebraic actions comes from three aspects:
firstly,  how to obtain exponential decay of matrix coefficients in general twisted spaces. The method used in
\cite{Spatzier1} requires that individual acting element acts ergodicly on the torus bundle. But this condition fails once $0$ weight appears. Secondly, for general partially hyperbolic actions, the stable and unstable foliations are no longer accessible. This means geometric method (the method in \cite{Damjanovic1}, \cite{Kononenko} and \cite{Zhenqi0}) can't be adapted to general cases. Thirdly, the smoothness of the solution to the cohomological equation followed from subelliptic
regularity theorem.  But this comes with three disadvantages: firstly, this requires that that the actions taken into account should satisfy condition ($\mathfrak B$); secondly, the solution of the cohomological equation loses at least half of regularity. Tame estimates (finite loss of regularity) for the solution  is important in dynamics, since it is closely related to obtain smooth action rigidity in dynamics, see \cite{Damjanovic2} and \cite{Damjanovic4}; thirdly, subelliptic
regularity theorem fails  for general Hilbert
spaces. For example, the methods in previous papers all fail if projection of the acting group to one simple factor of the semisimple part is trivial.

In this paper, we study the cohomological equation for general partially hyperbolic acting elements  and build up cocycle rigidity results for general higher-rank partially hyperbolic algebraic actions. We characterize the obstructions to solving the (twisted) cohomological equation, construct smooth solution and obtain the tame Sobolev estimates for the solution, i.e, there is finite loss of
regularity (with respect to Sobolev norms) between the coboundary and the
solution. As an application, we prove
the smooth (twisted) cocycle rigidity for any higher rank partially hyperbolic actions over $\GG$.
To prove these results, we introduce new ingredients from representation theory and obtain more elaborate information about estimates in
neutral directions. These results are of independent interest and have wide
applicability.

\subsection{Motivation} So far an effective approach to local differentiable rigidity is the ``geometric" method first introduced in \cite{Damjanovic1,Damjanovic2} to prove local smooth rigidity for generic restrictions in $SL(n,\RR)/\Gamma$ and $SL(n,\CC)/\Gamma$. This approach is based
on geometry and combinatorics of invariant foliations and using insights from
algebraic $K$-theory as an essential tool. The approach was further employed
in \cite{Zhenqi1}, \cite{Zhenqi2}, \cite{Zhenqi23} and \cite{Zhenqi0} for
extending cocycle rigidity and differentiable
 rigidity to most higher rank actions for symmetric space and twisted symmetric space examples satisfying the following
genuinely higher rank condition: the projection of the acting group to each
simple factor of the semisimple part contains a $\ZZ^2$ subgroup. The genuinely higher rank condition is necessary for
the application of geometric method. In many situations of interest, however this
condition is not present (for example, for the homogeneous space $SL(2,k)^n/\Gamma$, where $k=\RR$ or $\CC$).

One important application of the results in the present paper is that these results open a prospect of
proving a version of \emph{local differentiable rigidity} for general partially hyperbolic
actions. This should work as follows: by linearization of the conjugacy equation, we
get the corresponding linearized equation:
\begin{align}\label{for:1}
 \text{Ad}(\alpha)\Omega-\Omega\circ\alpha=R
\end{align}
where $\alpha$ is an $A$-algebraic action (the unperturbed action) and $R$ is the error between $\alpha$ and its
perturbation $\tilde{\alpha}$. If $\Omega$ is a solution for the linearized equation, or at least an approximate
solution, i.e., it solves the above equation with a small error with respect to $R$, then one may expect that the new
perturbation $\tilde{\alpha}^{(1)}$ defined by
 is much closer to $\alpha$ than $\tilde{\alpha}$. Carrying out the iteration process, one
may produce a smooth conjugacy between $\alpha$ and $\tilde{\alpha}$.

This method first appeared in \cite{Damjanovic4} to prove the differentiable rigidity of partially hyperbolic  but not  hyperbolic actions on torus. A  scheme similar to that of \cite{Damjanovic4} applies to  certain
parabolic cases, i.e.  homogeneous actions of unipotent   abelian groups in \cite{DK-parabolic}. To carry out the above scheme, the first task is to precisely describe the solution to the equation \eqref{for:1}, which is studied in Section \ref{sec:4}. Note that during the iteration process, the acting groups are not fixed, but vary in a small neighbourhood. So we need to obtain uniform estimates for these actions. Hence the results in the present paper are essential for successful application of the scheme to general partially hyperbolic actions in the future work, see \cite{Zhenqi4}.

\section{Background, definition, and statement of results}
\subsection{Preliminaries on cocycles} Let $\alpha:A\times
\mathcal{M}\rightarrow \mathcal{M}$ be an action of a topological group $A$ on a (compact)
manifold $\mathcal{M}$ by diffeomorphisms. For a topological group
$(Y,\ast)$ and a homomorphism $\psi:A\rightarrow \textrm{Aut}(Y)$, {\em a ($\psi$-twisted)-cocycle} over $\alpha$
is a continuous function $\beta : A\times E\rightarrow Y$
satisfying:
\begin{align}
\beta(ab, x) = \beta(a, \alpha(b, x))\ast\psi(a)\beta(b, x)
\end{align}
 for any $a, b \in A$. A (twisted) cocycle is
{\em cohomologous to a constant twisted cocycle} (cocycle not
depending on $x$) if there exists a constant (twisted) cocycle $s : A\rightarrow Y$ and a
continuous transfer map $H : \mathcal{M}\rightarrow Y$ such that for all $a
\in A $
\begin{align}\label{for:8}
 \beta(a, x) = H(\alpha(a, x))\ast s(a)\ast(\psi(a)H(x))^{-1}
\end{align}
\eqref{for:8} is called the cohomology equation.

In particular, a cocycle is a {\em coboundary} if it is cohomologous
to the trivial cocycle $\pi(a) = id_Y$, $a \in A$, i.e. if for all
$a \in A$ the following equation holds:
\begin{align}
 \beta(a, x) = H(\alpha(a, x))\ast (\psi(a)H(x))^{-1}.
\end{align}
For more detailed information on cocycles adapted to the present setting
see \cite{Damjanovic1} and \cite{Katok}.

In this paper we will only consider smooth $\CC^k$-valued cocycles over
algebraic partially hyperbolic actions on smooth manifolds. By taking component functions we may always assume that $\beta$ is valued on $\CC$. Further, by taking real and imaginary parts, we can extend the results for real valued cocycles as well.
Adapted to the settings in this paper, $A$ is isomorphic to $\ZZ^k$ or $\RR^k$ and the space $X=\GG/\Gamma$, if $\GG=G$
and $X=\GG/\Gamma\ltimes\ZZ^N$ if $\GG=G\ltimes\RR^N$, where $\Gamma$ is an irreducible torsion free lattice in $G$. A cocycle is called \emph{$\mathcal{H}^r$} if the map $\beta(a,\cdot)\in\mathcal{H}^r(L^2(\GG/\Gamma))$ for any $a\in A$, where $\mathcal{H}^r(L^2(\GG/\Gamma))$ is Sobolev space of order $r$ for the left regular representation of $\GG$ on $L^2(\GG/\Gamma)$. We can also define $\beta$ to be of class $C^r$.  We
also note that if the cocycle $\beta$ is cohomologous to a constant cocycle,
then the constant cocycle is given by $s(a)=\int_{\GG/\Gamma}\beta(a,x)dx.$

In what follows, $C$ will denote any constant that depends only
  on the given  group $\GG$. $C_{x,y,z,\cdots}$ will denote any constant that in addition to the
above depends also on parameters $x, y, z,\cdots$.

\subsection{Statement of the results}\label{sec:5} In this paper, $G$ denotes a higher-rank connected  semisimple Lie group with finite center and without compact factors.
Fix a maximal compact subgroup $K$ of $G$ and a right invariant, bi-$K$-invariant metric $d$ on $G$. Let $A_0$ be the $\RR$-split Cartan subgroup of $G$ admitting the Cartan decomposition $G=KA_0K$.

For a finite dimensional vector space $V$, a continuous representation $\rho: G\to GL(V )$ is called \emph{excellent} if
$\rho(G_i)$-fixed points in $V$ are $\{0\}$ for each simple factor $G_i$ of $G$. Set $\GG=G$ or $G\ltimes_\rho V$, where $\rho$ is excellent.
The
multiplication of elements in $G\ltimes_\rho V$ is defined by
\begin{align}\label{for:10}
(g_1,r_1)\cdot(g_2,r_2)=(g_1g_2,\rho(g_2^{-1})r_1+r_2).
\end{align}
For any $s\in G\ltimes_\rho V$, we have the decomposition $s=(g_s, v_s)$, where $g_s\in G$ and $v_s\in V$.  Then for any $s=(g_s, v_s)\in \GG$, we have
\begin{align}
 \text{Ad}(s)(X)&=\text{Ad}(g_s)(X)-\rho(g_s)d\rho(X)v_s,\quad \text{and}\notag\\
 \text{Ad}(s)(v)&=\rho(g_s)v.\label{for:19}
\end{align}
for any $X\in \text{Lie}(G)$ and $v\in V$.

Set $d_1\big((g,v),\,e\big)=d(g,e)+d_0(v,e)$, where $d_0$ is a metric on $V$. Fix an inner product on $\mathfrak{G}=\text{Lie}(\GG)$ (determined by $d$ or $d_1$). Let $\mathfrak{G}^1$ be the set of unit vectors in $\mathfrak{G}$.
\begin{definition}
An algebraic flow $\phi_{\RR}$ on $\GG$ (resp. an element $a\in \GG$ ) is called \emph{partially hyperbolic} if the spectrum of the group Ad$(\phi_t)$ (resp. Ad$(a)$) acting
on $\mathfrak{G}$ is not contained in $U(1)$.
\end{definition}
Our first two results characterize the obstructions to solving the cohomological
equation and obtain Sobolev estimates for the solution. The next theorem  shows that the $(\pi(a)-\lambda I)$-invariant distributions are the only obstructions to solving
the cohomological equation $\pi(a)f-\lambda f=h$ where $a\in \GG$ is partially hyperbolic and $\lambda\in U(1)$.
\begin{theorem}\label{th:6}
Suppose $(\pi,\,\mathcal{H})$ is a unitary representation  of $\GG$ such that the restriction of $\pi$ to any
simple factor of $G$ is isolated from the trivial representation (in the Fell topology) if $\GG=G$; or $\pi$ contains no non-trivial $V$-fixed vectors if $\GG=G\ltimes_\rho V$.  For any partially hyperbolic element $a$ and $\lambda\in U(1)$,
\begin{enumerate}
\item \label{for:3}if $f\in \mathcal{H}^1$ is a solution of the equation: $\pi(a)f-\lambda f=h$, then $\sum_{j=-\infty}^{+\infty}\lambda^{-(j+1)} \pi(a^j)h=0$ as a distribution.

\medskip
  \item \label{for:33} there exist constants $m_0>\sigma>1$ (only depending on $\GG$) such that for any $m\geq m_0$, there exists $\delta(m)>0$, such that for any $b\in \GG$ with $d(a,b)<\delta$, if $h\in \mathcal{H}^m$ satisfying $\sum_{j=-\infty}^{+\infty}\lambda^{-(j+1)} \pi(b^j)h=0$ as a distribution, then the equation \begin{align}\label{for:2}
       \pi(b)f-\lambda f=h
      \end{align}
      has a solution $f\in\mathcal{H}^{m-\sigma}$ and the following
estimate holds
\begin{align*}
\norm{f}_{m-\sigma}\leq C_{m,a} \norm{h}_{m}.
\end{align*}

 \item \label{for:35}if $h\in \mathcal{H}^\infty$ and $\mathcal{D}(h)=0$ for any $(\pi(a)-\lambda I)$-invariant distribution $\mathcal{D}\in \mathcal{H}^{-1}$, then the cohomological equation $\pi(a)f-\lambda f=h$, has a solution $f\in \mathcal{H}^\infty$.

\end{enumerate}
\end{theorem}
The next result is about the existence of common solution to the cohomological equations. Suppose $a,\,b\in\GG$ are partially hyperbolic and linearly independent and $m_0,\,\sigma$ as in Theorem \ref{th:6}.
\begin{theorem}\label{th:9} Suppose $(\pi,\,\mathcal{H})$ is a unitary representation  of $\GG$ such that the restriction of $\pi$ to any
simple factor of $G$ is isolated from the trivial representation  if $\GG=G$; or $\pi$ contains no non-trivial $V$-fixed vectors if $\GG=G\ltimes_\rho V$. For any $m\geq m_0+\sigma$, there exists $\delta(m)>0$, such that for any $a_1,\,b_1\in \GG$ with $d(a,a_1)+d(b,b_1)<\delta$ and $a_1b_1=b_1a_1$, if $f,\,h\in \mathcal{H}^m$
and satisfy the cocycle equation
\begin{align*}
  \pi(a_1)f-\lambda_1 f=\pi(b_1)h-\lambda_2 h
\end{align*}
where $\lambda_1,\,\lambda_2\in U(1)$, then the equations
\begin{align*}
  \pi(b_1)p-\lambda_2 p=f,\qquad \pi(a_1)p-\lambda_1 p=h
\end{align*}
have a common solution $p\in \mathcal{H}^{m-\sigma}$ satisfying
the Sobolev estimate
\begin{align*}
    \norm{p}_{m-\sigma}\leq C_{m,a,b}\max\{\norm{h}_{m}, \,\norm{f}_{m}\}.
\end{align*}
\end{theorem}

\begin{remark}
Results in Theorem \ref{th:6} and \ref{th:9} can be extended to partially hyperbolic flows correspondingly. See Corollary \ref{cor:3} and \ref{cor:4}.
\end{remark}
As an application, for the symmetric space examples and the twisted symmetric space examples we prove locally cocycle rigidity for any higher-rank partially hyperbolic action.  All relevant
definitions appear in Section \ref{regularrestrictions}.
\begin{theorem}\label{th:3}
Let $\alpha_{A}$ on $L\backslash\GG/\Gamma$ be an abelian higher-rank partially hyperbolic algebraic action of symmetric space examples or of the twisted symmetric space examples.
$\mathfrak{N}\subset\mathfrak{G}$ is the neutral distribution of $\alpha_{A}$ on $\GG/\Gamma$. Then there exist $p>s>0$ such that for any $m\geq p$:
\begin{enumerate}
  \item any $\mathcal{H}^m$-cocycle $\beta$ over $\alpha_{A}$ is cohomologous to a constant cocycle via a $\mathcal{H}^{m-s}$-transfer map;

\medskip
\item if $\beta$ is a $\mathcal{H}^m$-$\textrm{Ad}$-twisted cocycle taking values on $\mathfrak{N}$ over $\alpha_{A}$, then $\beta$ is cohomologous to a constant twisted
cocycle via a $\mathcal{H}^{m-s}$-transfer map.
\end{enumerate}
\end{theorem}

\noindent{\bf Acknowledgements.} I would like to thank  Roger Howe for discussion of matrix coefficients decay on
twisted symmetric spaces. Livio Flaminio suggested a
method of obtaining tame estimates in the centralizer direction
in a different setting  that inspired our arguments on that topic.
\section{Preliminaries on unitary representation theory}\label{sec:15}
\subsection{Sobolev space and elliptic regularity theorem}\label{sec:17} Let $\pi$ be a unitary representation of a Lie group $S$ with Lie algebra $\mathfrak{s}$ on a
Hilbert space $\mathcal{H}=\mathcal{H}(\pi)$.
\begin{definition}\label{de;1}
For $k\in\NN$, $\mathcal{H}^k(\pi)$ consists of all $v\in\mathcal{H}(\pi)$ such that the
$\mathcal{H}$-valued function $s\rightarrow \pi(s)v$ is of class $C^k$ ($\mathcal{H}^0=\mathcal{H}$). For $X\in\mathfrak{s}$, $d\pi(X)$ denotes the infinitesimal generator of the
one-parameter group of operators $t\rightarrow \pi(\exp tX)$, which acts on $\mathcal{H}$ as an essentially skew-adjoint operator. For any $v\in\mathcal{H}$, we also write $Xv:=d\pi(X)v$.
\end{definition}
We shall call $\mathcal{H}^k=\mathcal{H}^k(\pi)$ the space of $k$-times differentiable vectors for $\pi$ or the \emph{Sobolev space} of order $k$. The
following basic properties of these spaces can be found, e.g., in \cite{Nelson} and \cite{Goodman}:
\begin{enumerate}
  \item $\mathcal{H}^k=\bigcap_{m\leq k}D(d\pi(Y_{j_1})\cdots d\pi(Y_{j_m}))$, where $\{Y_j\}$ is a basis for $\mathfrak{s}$, and $D(T)$
denotes the domain of an operator on $\mathcal{H}$.

\medskip
  \item $\mathcal{H}^k$ is a Hilbert space, relative to the inner product
  \begin{align*}
    \langle v_1,\,v_2\rangle_{S,k}:&=\sum_{1\leq m\leq k}\langle Y_{j_1}\cdots Y_{j_m}v_1,\,Y_{j_1}\cdots Y_{j_m}v_2\rangle+\langle v_1,\,v_2\rangle
  \end{align*}
  \item The spaces $\mathcal{H}^k$ coincide with the completion of the
subspace $\mathcal{H}^\infty\subset\mathcal{H}$ of \emph{infinitely differentiable} vectors with respect to the norm
\begin{align*}
    \norm{v}_{S,k}=\bigl\{\norm{v}^2+\sum_{1\leq m\leq k}\norm{Y_{j_1}\cdots Y_{j_m}v}^2\bigl\}^{\frac{1}{2}}.
  \end{align*}
induced by the inner product in $(2)$. The subspace $\mathcal{H}^\infty$
coincides with the intersection of the spaces $\mathcal{H}^k$ for all $k\geq 0$.

\medskip
\item $\mathcal{H}^{-k}$, defined as the Hilbert space duals of
the spaces $\mathcal{H}^{k}$, are subspaces of the space $\mathcal{H}^{-\infty}$ of distributions, defined as the
dual space of $\mathcal{H}^\infty$.
  \end{enumerate}
We write $\norm{v}_{k}:=\norm{v}_{S,k}$ and $ \langle v_1,\,v_2\rangle_{k}:= \langle v_1,\,v_2\rangle_{S,k}$ if there is no confusion. Otherwise,
we use subscripts to emphasize that the regularity is measured with respect to $S$.

\subsection{Elliptic regularity theorem}
We list the well-known elliptic regularity theorem which will be frequently
used in this paper (see \cite[Chapter I, Corollary 6.5 and 6.6]{Robinson}):
\begin{theorem}\label{th:4}
Fix a basis $\{Y_j\}$ for $\mathfrak{s}$ and set $L_{2m}=\sum Y_j^{2m}$, $m\in\NN$. If $L_{2m}v\in\mathcal{H}$, then $v\in \mathcal{H}^{2m}$ with Sobolev estimate
\begin{align*}
    \norm{v}_{2m}\leq C_m(\norm{L_{2m}v}+\norm{v}),\qquad \forall\, m\in\NN
\end{align*}
where $C_m$ is a constant only dependent on $m$ and $\{Y_j\}$.
\end{theorem}
Suppose $\Gamma$ is a torsion-free cocompact lattice in $S$. Denote by $\Upsilon$ the regular representation of $S$ on $\mathcal{H}(\Upsilon)=L^2(S/\Gamma)$. Then we have the following subelliptic regularity theorem (see \cite{Spatzier2}):
\begin{theorem}\label{th:5}
Fix $\{Y_j\}$ in $\mathfrak{s}$ such that commutators of $Y_j$ of length at most $r$ span $\mathfrak{s}$. Also set $L_{2m}=\sum Y_j^{2m}$, $m\in\NN$. Suppose $f\in\mathcal{H}(\Upsilon)$ or  $f\in \mathcal{H}^{-\infty}$. If $L_{2m}f\in \mathcal{H}(\Upsilon)$ for any $m\in\NN$, then $f\in \mathcal{H}^\infty(\Upsilon)$ and satisfies
\begin{align}\label{for;1}
\norm{f}_{\frac{2m}{r}-1}\leq
C_m(\norm{L_{2m}f}+\norm{f}),\qquad \forall\, m\in\NN
\end{align}
where $C_m$ is a constant only dependent on $m$ and $\{Y_j\}$.
\end{theorem}
\begin{remark}
The elliptic regularity theorem is a general property, while the subelliptic regularity theorem only has local versions on manifolds (see \cite{Rothschild}). More precisely, the proof of the above theorem is based on the following general subelliptic regularity theorem: for any $x\in S/\Gamma$ there is an open neighbourhood $V$ containing $x$ such that if $f$ and $L_{2m}f$ are both in $L^2(V)$, then
\begin{align*}
\norm{f}_{\frac{2m}{r}-1}\leq
C_{m,V}(\norm{L_{2m}f}+\norm{f}),\qquad \forall\, m\in\NN
\end{align*}
where $C_{m,V}$ is a constant only dependent on $m$, $\{Y_j\}$ and $V$. In other words, we can only get a local version of \eqref{for;1} on the manifold $S/\Gamma$. Then the compactness of $S/\Gamma$ is essential to the existence of a uniform constant $C_m$ for the global Sobolev estimates.
\end{remark}
\subsection{Exponential matrix coefficients decay}
\begin{definition}
Let $\pi$ be a unitary
representation of $S$ on a Hilbert space $\mathcal{H}$. Say that a
vector $v\in\mathcal{H}$ is $\delta$-Lipschitz  if
\begin{align*}
\delta=\sup_{g\in
G-\{e\}}\frac{\norm{\pi(g)v-v}}{\text{dist}(e,g)}<\infty;
\end{align*}
we will refer to the number $\delta$ as to the $\delta$-Lipschitz
coefficient of $v$, and say that the vector $v$ is
$\delta$-Lipschitz.
\end{definition}
If $S$ is semisimple without compact factors and with finite center, Kleinbock and Margulis (see \cite[appendix]{Kleinbock}) extended the matrix coefficient decay result about smooth vectors in \cite{Spatzier1}
to Lipschitz vectors. Fix a maximal compact subgroup $K$ of $S$ and a Riemannian metric $d$ on $S$ which is bi-invariant with respect to $K$.
\begin{theorem}[Kleinbock and Margulis]\label{th:2} Let $(\pi,\,\mathcal{H})$ be a unitary representation of
$S$  such that the restriction of $\pi$ to any
simple factor of $S$ is isolated from the trivial representation. Then there exist constants $\gamma,E>0$,
dependent only on $S$ such that if $v_i\in \mathcal{H}$, $i=1,2$, be
$\delta_i$-Lipschitz vectors then for any $g\in S$
\begin{align*}
\abs{\langle\pi(g)v_1,v_2\rangle}&\leq
(E\norm{v_1}\norm{v_2}+\delta_1\norm{v_2}+\delta_2\norm{v_1}+\delta_1\delta_2)e^{-\gamma d(e,g)}.
\end{align*}
\end{theorem}

\begin{remark}\label{re:5}
If $\GG=G\ltimes_\rho V$, then it also follows from Theorem $1.2$ of \cite{Zhenqi23} that for any unitary representation $\pi$ of
$\GG$ without $V$-fixed vectors, its restriction to any simple factor of $G$ is isolated from the trivial representation. The implies that
the above theorem applies for the restriction of $\pi$ to $G$.
\end{remark}

\section{Solution of the twisted coboundary}\label{section:linear}
Throughout this part, $(\pi,\,\mathcal{H})$ always denotes a unitary representation  of $\GG$ such that the restriction of $\pi$ to any
simple factor of $G$ is isolated from the trivial representation  if $\GG=G$; or $\pi$ contains no non-trivial $V$-fixed vectors if $\GG=G\ltimes_\rho V$.

The subsequent discussion will be devoted to the proof of Lemma \ref{le:8}. Typically
differences between the cases $\GG=G$ and $\GG=G\ltimes_\rho V$ are minimal, and usually appear at the level of
notations. However, the case $\GG=G\ltimes_\rho V$ requires a separate argument in order to obtain exponential decay of matrix coefficients.

First we define obstructions to solvability of the twisted coboundary  equation,  for a single element  and show that
vanishing of those obstructions implies solvability of  the equation
with tame estimates with respect to Sobolev norms. The latter property  is an
instance of {\em cohomological stability}, the  notion  first
defined in \cite{Katok-constructions}. The scheme of the proof is as follows:

\begin{enumerate}
\item We note  that $f$ is the solution of the equation \eqref{for:2}
if and only if $\pi(s)f$ is the solution of the equation
\begin{align}\label{for:14}
 \pi(sbs^{-1})\pi(s)f-\lambda \pi(s)f=\pi(s)h;
\end{align}
and $\sum_{j=-\infty}^{+\infty}\lambda^{-(j+1)} \pi(b^j)h=0$ as a distribution if and only if
\begin{align*}
 \sum_{j=-\infty}^{+\infty}\lambda^{-(j+1)} \pi((sbs^{-1})^j)(\pi(s)h)=0
\end{align*}
as a distribution. Then instead of solving the equation \eqref{for:2}, we solve the equation \eqref{for:14} for a well chosen $s$ (see Section \ref{sec:1}).

\item The crucial step in proving Theorem \ref{th:6} is Lemma \ref{le:8}. Since generally, condition $(\mathfrak B)$ fails, we study the restricted representation $\pi'$ of $\pi$ on $\GG'$ (see Section \ref{sec:2}) instead of $\pi$. Note that condition $(\mathfrak B)$ holds on $\GG'$.

\item    Decay estimates for matrix coefficients imply existence of two distribution solutions obtained
by iteration in positive and negative directions: one of those solutions is  differentiable along  stable directions and the other along unstable directions.
\item Vanishing of the obstructions implies  that those distribution solutions coincide. Since solution along the stable and unstable directions  is given by explicit  exponentially converging ``telescoping sums'',  they can be differentiated without  loss of regularity. Up to this point  the proof follows the same general  scheme as in \cite{Spatzier1} although we obtain more elaborate information about estimates in other directions.

\item Remaining directions in $\GG'$ with Ad eigenvalues sufficiently close to $1$ for the acting element; hence  derivatives of all orders in those direction have every slow increasing speed.
Tame estimates follow from that and from the fact that  those vector-fields can be expressed as polynomial  of hyperbolic ones, i.e. from condition   $(\mathfrak B)$ on $\GG'$.

\item Note that the derivatives of all orders in directions outside $\GG'$  are still distributions for $\pi'$ and satisfy the solvability condition. Then estimates for these directions follow from previous steps.
\end{enumerate}

\subsection{Spectral space decomposition}\label{sec:2} If $\GG=G$, for any $s\in G$ under the adjoint representation of $G$, $\text{Ad}(s)$ has the decomposition:
\begin{align}\label{for:1201}
  \text{Ad}(s)= \exp(Z_s)\exp(X_s)\exp(Y_s)
\end{align}
for $3$ commuting elements, where $Z_s$ is compact, $X_s$ is $\RR$-semisimple, $Y_s$ is nilpotent (see, eg, \cite[Proposition 2]{goto78}).

The decomposition \eqref{for:1201} also implies that we have the corresponding decomposition for $a$:
\begin{align}\label{for:9}
 s=k_s x_s n_s
\end{align}
where $x_a=\exp(X_a)\in G$ is $\RR$-semisimple, $k_a=\exp(Z_s)\in G$ is compact and $n_s=\exp(Y_a)\in G$ is nilpotent. It is clear that $x_a$ and $n_a$ commute and $k_sx_s=x_sk_sz_1$ and $k_sn_s=n_sk_sz_2$, where $z_1,\,z_2\in Z(G)$. Since $Z(G)$ is finite, there exists $n\in N$ such that $k_sx_s^nk_s^{-1}=x_s^n$, which implies that $\text{Ad}(k_s)(X_s)=X_s$. This shows that $k_s$ and $x_s$ commute. Similarly, we get $k_sn_s=n_sk_s$.

If $s$ is partially hyperbolic, $x_s$ is non-trivial. If $\GG=G\ltimes_\rho V$, For any $s=(g_s, v_s)\in\GG$, if $s$ is partially hyperbolic, $x_{g_s}$ is non-trivial.

For any partially hyperbolic element $s$, the Lie algebra
$\mathfrak{G}$ of $\GG$ has the eigenspace decomposition for $\text{Ad}(x_s)$ or $\text{Ad}(x_{g_s})$:
\begin{align}\label{for:5}
\mathfrak{G}=\sum_{\mu\in\Delta(s)}\mathfrak{g}_\mu(s)
\end{align}
where $\Delta(s)$ is the set of eigenvalues and $\mathfrak{g}_\mu(s)$ is the eigenspace for eigenvalue $\mu$. We note that the eigenvalues
of $\text{Ad}(s)$ are determined by those of $\text{Ad}(x_s)$ or $\text{Ad}(x_{g_s})$ up to some elements in $U(1)$.

Let $\mathbf{g}$ be the subalgebra generated by all $\mathfrak{g}_\mu$, $\mu\neq 1$. Then:
\begin{lemma}\label{le:2}
\begin{enumerate}
  \item \label{for:71}$\mathbf{g}$ is an ideal in $\mathfrak{G}$.
  \smallskip
  \item \label{for:73} If $\GG=G\ltimes V$, then $V\subset \mathbf{g}$.
\end{enumerate}
\end{lemma}
\begin{proof}
\eqref{for:71} follows directly form the fact that $[\mathfrak{g}_{\mu_1},\mathfrak{g}_{\mu_2}]\subseteq \mathfrak{g}_{\mu_1\mu_2}$. Let
$\GG'(s)$ be the connected subgroup with Lie algebra $\mathbf{g}$. The semisimple part of $\GG'(s)$ is an almost direct product of simple factors of the semisimple part of $\GG$. If $\GG=G\ltimes V$, by complete reducibility for representations of
semisimple groups, there is a decomposition
of
\begin{align*}
 V=\bigoplus_{i\in I}V_i,
\end{align*}
for restricted representation $\rho$ on $\GG'(s)$ such that $\rho$ is irreducible on each $V_i$. Since $\rho$ is excellent (see Section \ref{sec:5}), the restricted representation is non-trivial on each $V_i$. By \eqref{for:19} we see that each $V_i$ is contained in $\mathbf{g}$. Then we get \eqref{for:73} immediately.
\end{proof}
\begin{remark}
For a fixed partially hyperbolic element $a\in\GG$, let $\mathbf{g}$ be the subalgebra generated by all $\mathfrak{g}_\mu$, $\mu\neq 1$. Let
$\GG'=\GG'(a)$ be the connected subgroup with Lie algebra $\mathbf{g}$. $\GG'$ is dependent on $a$; furthermore, even if $b$ is close to $a$, $\GG'(b)$ is not necessarily to be equal to $\GG'(a)$. For example, if $\GG$ is semisimple and $a$ is inside a simple factor $G_1$ of $\GG$, then for $b$ close enough to $a$, $\GG'(b)$ can be any product of simple factors of $\GG$ containing $G_1$. But if we consider the subalgebra generated by the $\mathfrak{g}_\mu(b)$, for those $\mu$ sufficiently close to $\Delta(a)\setminus 1$, then it is still $\mathbf{g}$ (see Lemma \ref{le:3}).
\end{remark}
Next, we study the spectral space decomposition of elements sufficiently close to $a$. Suppose $b$ is sufficiently close to $a$. We will need to consider the eigenspaces for $\text{Ad}(x_{b})$ or $\text{Ad}(x_{g_b})$. In next section, we want to study the restricted representation $\pi$ on $\GG'$.  Since $b$ is probably not contained in $\GG'$ as explained in previous part, we need to consider the decomposition of $x_{b}$ or $x_{g_b}$ instead. We have a direct sum decomposition of
\begin{align*}
\text{Lie}(G)=\mathbf{g}_1\oplus \mathbf{g}_2.
\end{align*}
where $\mathbf{g}_1$ is the Lie algebra of the semisimple part of $\GG'$. Then we have the decomposition:
\begin{align*}
  x_{b}=x_{b,1}x_{b,2}\quad\text{ or }\quad x_{g_b}=x_{g_b,1}x_{g_b,2},
\end{align*}
where $x_{b,i}$ and $x_{g_b,i}$, $i=1,\,2$ is in the connected subgroup with Lie algebra $\mathbf{g}_i$. Note that $x_{b,i}$ or $x_{g_b,i}$, $i=1,\,2$ commute with $x_{a}$ or $x_{g_a}$ respectively. It is clear that $x_{b,2}$ or $x_{g_b,2}$ are sufficiently close to identity if $b$ is sufficiently close to $a$.

We have the eigenspace decomposition for $\text{Ad}(x_{b,1})$ or $\text{Ad}(x_{g_b,1})$:
\begin{align*}
\mathfrak{G}=\sum_{\mu\in\Delta'(b)}\mathfrak{l}_\mu.
\end{align*}
where $\Delta'(b)$ is the set of eigenvalues of $\text{Ad}(x_{b,1})$ or $\text{Ad}(x_{g_b,1})$ and $\mathfrak{l}_\mu$ is the eigenspace for eigenvalue $\mu$. Let $\Delta''(b)=\{\mu\in \Delta'(b): \mu\in (\nu-\epsilon,\nu+\epsilon)\text{ where }\nu\in \Delta(a)\backslash 1\}$ for sufficiently  small $\epsilon$. In fact, $\Delta''(b)$ excludes those eigenvalues sufficiently close to $1$. Then:
\begin{lemma}\label{le:3}
Let $\mathbf{g}'$ be the subalgebra generated by all $\mathfrak{l}_\mu$, $\mu\in \Delta''(b)$. Then $\mathbf{g}'$ is an ideal and $\mathbf{g}'=\mathbf{g}$.
\end{lemma}
\begin{proof} To prove that $\mathbf{g}'$ is an ideal it suffices to shows that for any $\mu_1\in \Delta''(b)$ and $\mu_2\in \Delta'(b)\backslash \Delta''(b)$,
$[\mathfrak{g}_{\mu_1},\mathfrak{g}_{\mu_2}]\subseteq \mathfrak{g}_{\mu'}$ for some $\mu'\in \Delta''(b)$. We note that $[\mathfrak{g}_{\mu_1},\mathfrak{g}_{\mu_2}]\subseteq \mathfrak{g}_{\mu_1\mu_2}$.  By assumption, $\mu_1\mu_2\in \Delta''(b)$ since $\mu_1\mu_2$ is ``far away" from $1$ by assumption. If $\GG=G$, then it is clear that $\mathbf{g}'=\mathbf{g}$. If $\GG=G\ltimes V$, $\mathbf{g}'=\mathbf{g}$ have the same semisimple part by previous arguments. Lemma \ref{le:2} shows that both $\mathbf{g}'$ and $\mathbf{g}$ contain $V$. Then the result follows immediately.
\end{proof}
Hence we have the decomposition:
\begin{align}\label{for:74}
 \mathfrak{G}=\mathbf{g}\oplus\mathbf{g}_2=\mathbf{g}\oplus\sum_{\mu\in\Delta'''(b)}\mathfrak{l}_\mu,
\end{align}
where $\Delta'''(b)\subset \Delta'(b)\backslash \Delta''(b)$. Note that elements in $\Delta'''(b)$ are sufficiently close to $1$.

For each $\mu\in\Delta'(b)$, set $X^\mu_1,\cdots,X^\mu_{\dim\mathfrak{l}_\mu}\in \mathfrak{G}_1$ to be a basis of
$\mathfrak{l}_\mu$. Above lemma shows that there exists $r(a)>0$ such that for any $b$ sufficiently close to $a$, $X^\mu_j$, where $\mu\in\Delta''(b)$, $1\leq j\leq \dim\mathfrak{l}_\mu$ as well as their commutators of length no larger than $r(a)$ span $\mathbf{g}$.

Let $p_j$ be polynomials with degree  no greater  than
$r(a)$ such that the linear span of the set
\begin{align}\label{for:4}
  \big\{p_j(X^{\mu_{l(1)}}_{j(1)},\cdots,X^{\mu_{l(i)}}_{j(i)}),\,X^\mu_t\big\}
\end{align}
where $\mu_{l(k)}\in\Delta''(b)$ for all $1\leq k\leq i$ and $1\leq t\leq\dim\mathfrak{l}_\mu$, $\mu\in \Delta''(b)$, generate $\mathbf{g}$.
Here we note that after substituting elements in $\mathbf{g}$, these $p_j$ take values in the universal enveloping algebra
$U(\mathbf{g})$.

\subsection{Conjugation in semidirect product}\label{sec:1} If $\GG=G\ltimes_\rho V$, we note that $I-\rho(g_b)$ restricted on the subspace $W=\sum_{\mu\in \Delta''(b)} \mathfrak{g}_\mu\bigcap V$ is invertible. Denote by $(I-\rho(g_b))^{-1}\mid_W$ the inverse map on $W$.  Then we have a decomposition:
\begin{align*}
  v_b=\underbrace{\sum_{\mu\in \Delta''(b)} v_{b,\mu}}_{v_{b}(1)}+\underbrace{\sum_{\mu\in\Delta'(b)\setminus\Delta''(b)} v_{b,\mu}}_{v_{b}(2)}
\end{align*}
where $v_{b,\mu}\in\mathfrak{g}_\mu\bigcap V$. Then $u=\rho(g_b)(v_{b}(1))\in W$. Set $u_b=-(I-\rho(g_b))^{-1}|_W(\rho(g_b)u)$ and $s=(g_b,u_b)$. Then by easy computation we have
\begin{align*}
  sbs^{-1}=\big(g_b,\rho(g_b)v_{b}'\big),
\end{align*}
where $\rho(g_b)v_{b}'\in \sum_{\mu\in\Delta'(b)\setminus\Delta''(b)} v_{b,\mu}\bigcap V$.

Note that the norms of such $s$ are uniformly bounded for $b$ sufficiently  close to $a$. By discussion at the beginning of Section \ref{section:linear} (we see that conjugation by $s$ doesn't affect the conclusions in Theorem \ref{th:6}), we can just assume that $b$ has the decomposition: $b=(g_b, v_b)$, $v_b\in \sum_{\mu\in\Delta'(b)\setminus\Delta''(b)} v_{b,\mu}$. Then we can write $b=(x_{g_b,1},0)(x_{g_b,2}k_{g_b}n_{g_b},v_b)$. Since the eigenvalues of $\text{Ad}(x_{g_b,1})$
on $\Delta'(b)\setminus\Delta''(b)$ are sufficiently close to $1$, we have
\begin{align}\label{for:47}
 \norm{\text{Ad}(x_{g_b,1}^j)(v_b)}\leq C(1+\epsilon)^{\abs{j}}\norm{v_b},\qquad \forall j\in\ZZ.
\end{align}
Using \eqref{for:19}, we have
\begin{align}\label{for:46}
 b^j&=\big(g_b^j, \sum_{i=0}^{j-1}\rho(g_b^{-i})v_b\big)=(x_{g_b,1}^j,0)\big(x^j_{g_b,2}k^j_{g_b}n^j_{g_b}, \sum_{i=0}^{j-1}\rho(g_b^{-i})v_b\big)\notag\\
 &=(x_{g_b,1}^j,0)\big(x^j_{g_b,2}k^j_{g_b}n^j_{g_b}, \sum_{i=0}^{j-1}\text{Ad}(g_b^{-i})v_b\big)
\end{align}
Let $y_{b,j}=x_{b,2}^jk_b^jn_b^j$ if $\GG=G$ or $y_{b,j}=\big(x^j_{g_b,2}k^j_{g_b}n^j_{g_b}, \sum_{i=0}^{j-1}\text{Ad}(g_b^{-i})v_b\big)$ if $\GG=G\ltimes_\rho V$. If there is no confusion, we abuse $x_{b,1}$ and $x_{g_b,1}$ for simplicity.

\subsection{Twisted cohomological stability}\label{sec:cohstability}

Fix a partially hyperbolic element $a\in\GG$. It is clear that $\pi$ is also a unitary representation for $\GG'(a)$. We use $(\pi',\,\mathcal{H})$ to denote the restricted representation. Define $r_0$ to be the minimal
positive integer satisfying:
\begin{align}\label{for:31}
&\chi\cdot\phi^{r_0}>1\text{ if }\phi>1,\quad\text{and}\quad\chi\cdot \phi^{r_0}<1 \text{ if }\phi<1
\end{align}
for any $\chi,\phi\in\Delta(a)$.

Let $m_0=\max\{r_0,r(a)\}$, where $r(a)$ is defined \eqref{for:4}. We list the following lemma  which is very important for the sequel.
\begin{lemma}\label{le:8}
Suppose $b\in\GG$ is sufficiently close to $a$. Also suppose $\lambda\in\CC$ with $\abs{\lambda}$ sufficiently close to $1$. Then:
\begin{enumerate}
 \item\label{for:64} For any $j\in\ZZ$, $v\in \mathcal{H}^1(\pi)$ and $u\in \mathcal{H}^1(\pi')$ we have
\begin{align*}
\abs{\langle\pi(b^j)v,u\rangle}\notag&\leq C(1+\epsilon)^{\abs{j}}(\abs{j}+1)^{\dim \mathfrak{G}}\big(\norm{v}\norm{u}+\norm{u}\norm{v}_{\GG,1}\\
&+\norm{v}\norm{u}_{\GG',1}+\norm{u}_{\GG',1}\norm{v}_{\GG,1}\big)e^{-\gamma\abs{j}l(b)},
\end{align*}
where $l(b)=\frac{1}{2}\sum_{\mu\in\Delta(b)}\abs{\log\mu}>0$, $\gamma>0$ is a constant only dependent on $\GG$ and $\epsilon\geq 0$ is sufficiently small.

\medskip
  \item \label{for:60} Suppose $v=\{v_j\}$, $j\in\ZZ$ is a sequence in $\mathcal{H}^1(\pi)$ satisfying $\norm{v_j}_{\GG,1}\leq P(\abs{j})$, where $P$ a polynomial, then
      \begin{align*}
       D_v^{\binom{+}{-}}(b)=\binom{-}{+}\sum_{\binom{j\geq0}{j\leq
-1}}\lambda^{-(j+1)}\pi(b^j)v_j
      \end{align*}
  are distributions in $\pi'$.

\medskip
 \item \label{for:61} Suppose $m>m_0+3$ and suppose $v=\{v_j\}$, $j\in\ZZ$ is a sequence in $\mathcal{H}^m(\pi)$ satisfying $\norm{v_j}_{\GG,s}\leq P_s(\abs{j})\norm{h}_{\GG,s}$, where $P_s$ are polynomials, $0\leq s\leq m$ and $h$ is a vector in $\mathcal{H}^m(\pi)$. There exists $\delta(m)>0$, such that for any $b\in \GG$ with $d(a,b)<\delta$, if $f_v\stackrel{\text{def}}{=}D_v^+(b)=D_v^-(b)$ as distributions in $\pi'$. Then $f_v\in \mathcal{H}^{m-m_0-3}(\pi')$ and the following
estimate holds
\begin{align*}
\norm{f_v}_{\GG',m-m_0-3}\leq C_{m,p_1,\cdots,P_m} \norm{h}_{\GG,m}.
\end{align*}

\end{enumerate}

\end{lemma}
\begin{remark}
In \eqref{for:61}, $m$ is dependent on the closeness of $b$ and $a$ for general partially hyperbolic element $a$. If $\GG=G$ and $a$ is regular, i.e., $\chi(a)\neq0$ for any roots of $G$, then $m$ is independent on the closeness.
\end{remark}

\begin{proof} Let $G'$ denote the semisimple part of $\GG'$. If $\GG=G$, then it is clear that the restriction of $\pi'$ to any
simple factor of $G'$ is isolated from the trivial representation; if $\GG=G\ltimes_\rho V$, then it also follows from Theorem $1.2$ of \cite{Zhenqi23} that the restriction of $\pi'$ to any simple factor of $G'$ is isolated from the trivial representation. The above arguments justify the application of Theorem \ref{th:2}
for the restricted representation of $(\pi',\mathcal{H})$ on $G'$.

\medskip
\emph{\textbf{Proof of \eqref{for:64}}}. Recall notations in Section \ref{sec:1}. By Theorem \ref{th:2} for any $v\in \mathcal{H}^1(\pi)$ and $u\in\mathcal{H}^1(\pi')$  we have,
\begin{align}\label{for:54}
\abs{\langle\pi(b^j)v,u\rangle}&=\abs{\langle\pi'(x_{b,1}^j)(\pi(y_{b,j})v),u\rangle}\notag\\
&\leq
C(\norm{\pi(y_{b,j})v}\norm{u}+\delta_j\norm{u}+\delta\norm{\pi(y_{b,j})v_j}+\delta_j\delta)e^{-\gamma_1 d(e,x_{b,1}^j)}\notag\\
&=C(\norm{v}\norm{u}+\delta_j\norm{u}+\delta\norm{v}+\delta_j\delta)e^{-\gamma_1 d(e,x_{b,1}^j)}
\end{align}
where $\delta=\norm{u}_{\GG',1}$, $\delta_j=\norm{\pi(y_{b,j})v}_{\GG',1}$ and $\gamma_1>0$ is a constant only dependent on $\GG$.

Since $x_{b,1}$ is partially hyperbolic and conjugated to an element in $A_0$, then
\begin{align}\label{for:50}
d(e,x_{b,1}^j)\geq C\abs{j}l(b)
\end{align}
where $l(b)=\frac{1}{2}\sum_{\mu\in\Delta(b)}\abs{\log\mu}>0$. Here we use $\Delta(b)$ instead of $\Delta'(b)$ since $x_{b,1}$ is sufficiently close to $x_b$.

If $\GG=G$, since $k_b$ is compact, $n_b$ is unipotent, $x_{b_2}$ is sufficiently close to identity and they are commuting, then for any $Y\in \mathfrak{G}$ we have
\begin{align}\label{for:21}
  \norm{\text{Ad}(y_{b,j})Y}\leq C(1+\epsilon)^{\abs{j}}(\abs{j}+1)^{\dim \mathfrak{G}}\norm{Y},\qquad \forall j\in\ZZ,
\end{align}
where $\epsilon\geq0$ is sufficiently small.

If $\GG=G\ltimes_\rho V$, first, we note that in the expression of $y_{b,j}$ (see \eqref{for:46}),
\begin{align*}
 \text{Ad}(g_b^{-i})v_b=\text{Ad}(x^{-i}_{g_b,2}k^{-i}_{g_b}n^{-i}_{g_b}x_{g_b,1}^{-i})v_b.
\end{align*}
Then by using \eqref{for:47} and noting that $\text{Ad}(x_{g_b,2})$ is sufficiently close to identity, we see that estimates in \eqref{for:21} still hold. Note that in \eqref{for:50} and \eqref{for:21} we can take uniform $C$ for all $b$ sufficiently close to $a$.

Also note that for any $Y\in \mathfrak{G}$
\begin{align}\label{for:58}
  Y\pi(y_{b,j})v=\pi(y_{b,j})\big(\text{Ad}(y_{b,j}^{-1})Y\big)v.
\end{align}
Then it follows that
\begin{align}\label{for:41}
\delta_j\leq C(1+\epsilon)^{\abs{j}}(\abs{j}+1)^{\dim \mathfrak{G}}\norm{v}_{\GG,1}.
\end{align}
Then the estimate for $\abs{\langle\pi(b^j)v,u\rangle}$ follows directly from \eqref{for:54}, \eqref{for:50} and \eqref{for:41}.

\medskip
\emph{\textbf{Proof of \eqref{for:60}}}. It follows from previous result that
\begin{align*}
&\sum_{j=-\infty}^\infty\abs{\langle\lambda^{-(j+1)}\pi(b^j)v_j,u\rangle}=\sum_{j=-\infty}^\infty\abs{\lambda}^{-(j+1)}\cdot\abs{\langle\pi(b^j)v_j,u\rangle}\\
&\leq C\sum_{j=-\infty}^\infty\abs{\lambda}^{-(j+1)}(1+\epsilon)^{\abs{j}}(\norm{v_j}\norm{u}+\delta_j\norm{u}+\delta\norm{v_j}+\delta_j\delta)e^{-\gamma d(e,x_{b,1}^j)}\\
&\leq C\sum_{j=-\infty}^\infty\abs{\lambda}^{-(j+1)}(1+\epsilon)^{\abs{j}} (\abs{j}+1)^{\dim \mathfrak{G}}\norm{v_j}_{\GG,1}\norm{u}_{\GG',1}e^{-C\gamma l(b)\abs{j}}\\
&\leq C\sum_{j=-\infty}^\infty \abs{\lambda}^{-(j+1)}(1+\epsilon)^{\abs{j}}(\abs{j}+1)^{\dim \mathfrak{G}}P(\abs{j})\norm{u}_{\GG',1}e^{-C\gamma l(b)\abs{j}}\\
&<+\infty
\end{align*}
This shows that $D_v^{\binom{+}{-}}(b)$ are distributions in $\pi'$.
\medskip

\noindent\emph{\textbf{Proof of \eqref{for:61}}} We will show differentiability of $f_v$ by using both of its forms. First, we show differentiability of $f_v$ in $X^\mu_i$, $\mu\in\Delta''(b)$.

For any $X^\mu_i$, $1\leq i\leq \dim\mathfrak{l}_\mu$, if $\mu>1$, we may use the $D_v^{+}$ form to obtain the following bound on $s$'th derivative
\begin{align}\label{for:30}
&\sum_{j=0}^\infty (X^\mu_i)^s\big(\lambda^{-(j+1)} \pi(b^j)h\big)=\sum_{j=0}^\infty \lambda^{-(j+1)}\mu^{-js}\pi_1(x_{b,1}^j)(X^\mu_i)^s\big( \pi(y_{b,j})h\big)\notag\\
&=\sum_{j=0}^\infty \lambda^{-(j+1)}\mu^{-js}\pi_1(x_{b,1}^j)\big(\pi(y_{b,j})(Z_j)^s h\big)
\end{align}
where $Z_j=\textrm{Ad}(y_{b,j}^{-1})(X^\mu_i)$ for all $1\leq s\leq m$.

By using \eqref{for:21} we see that the left-hand  side of
\eqref{for:30} converges absolutely in $\mathcal{H}$ with estimates
\begin{align}\label{for:37}
\norm{(X^\mu_i)^sD_v^+}&\leq \sum_{j=0}^\infty \mu^{-js}\norm{\lambda^{-(j+1)}\pi_1(x_{b,1}^j)\big(\pi(y_{b,j})(Z_j)^s v_j\big)}\notag\\
&=\sum_{j=0}^\infty \mu^{-js}\abs{\lambda}^{-(j+1)}\norm{(Z_j)^s v_j}\notag\\
&\leq \sum_{j=0}^\infty C\mu^{-js}\abs{\lambda}^{-(j+1)}(1+\epsilon)^{\abs{j}}(\abs{j}+1)^{\dim \mathfrak{G}}P_s(\abs{j})\norm{h}_{\GG,s}\notag\\
&\leq C_{s,P_s}\norm{h}_{\GG,s}.
\end{align}
Similarly, if $\mu<1$ using the form $f_v=D_v^{-}$,
the estimates
\begin{align}\label{for:22}
\norm{(X^\mu_i)^sD_v^-}\leq C_{s,P_s}\norm{h}_{\GG,s}.
\end{align}
holds if $1\leq s\leq m$.

Before we show differentiability in other directions, we obtain tame estimates for $X^\mu_i f_v$, $\mu\in\Delta''(b)$ instead, which is important for that purpose.

Fix $X^\mu_i$, $\mu\in\Delta''(b)$, $1\leq i\leq \dim\mathfrak{g}_\mu$. \eqref{for:37} and \eqref{for:22} show that
\begin{align}\label{for:24}
\norm{(X^\mu_i)f_v}\leq C_{P_1}\norm{h}_{\GG,1}.
\end{align}
If $\mu>1$, for any $X^\chi_j$ with $\chi\in \Delta''(b)$ and $\chi>1$ or $\chi\in\Delta'(b)\backslash\Delta''(b)$, we have
\begin{align*}
(X^\chi_j)^s(X^\mu_iD_v^+)=-\sum_{\ell=0}^\infty \lambda^{-(\ell+1)}(\chi^{s}\mu)^{-\ell}\pi_1(x_{b,1}^\ell)\big(\pi(y_{b,\ell})(Z_\ell)^{s} Y_\ell v_\ell\big)
\end{align*}
where  $Z_\ell=\textrm{Ad}(y_{b,\ell}^{-1})(X^\chi_j)$ and $Y_\ell=\textrm{Ad}(y_{b,\ell}^{-1})(X^\mu_i)$ for all $1\leq s\leq m-1$.

Since $\chi$ is succinctly close to $1$ if $\chi\in\Delta'(b)\backslash\Delta''(b)$, this together with \eqref{for:21} shows that the estimates
\begin{align}\label{for:23}
\norm{(X^\chi_j)^s(X^\mu_iD_v^+)}\leq C_{s,P_{s+1}} \norm{h}_{\GG,s+1}.
\end{align}
holds if $1\leq s\leq m-1$ for any $X^\chi_j$ with $\chi>1$ or $\chi\in\Delta'_b\backslash\Delta''(b)$.

For any $X^\chi_j$, $\chi\in \Delta''(b)$ with $\chi<1$, since $b$ is sufficiently to $a$, similar to \eqref{for:31} we also have $\chi^{s}\mu<1$ for any $r_0\leq s$. Then  we have
\begin{align*}
(X^\chi_j)^s(X^\mu_iD_v^-)=\sum_{\ell=-1}^{-\infty} \lambda^{-(\ell+1)}(\chi^{s}\mu)^{-\ell}\pi_1(x_{b,1}^\ell)\big(\pi(y_{b,\ell})(Z_\ell)^{s} Y_\ell v_\ell\big),
\end{align*}
for $r_0\leq s\leq m-1$, where $Z_\ell=\textrm{Ad}(y_{b,\ell}^{-1})(X^\chi_j)$ and $Y_\ell=\textrm{Ad}(y_{b,\ell}^{-1})(X^\mu_i)$.

Together with \eqref{for:21} it follows that
\begin{align}\label{for:25}
\norm{(X^\chi_j)^s(X^\mu_iD_v^-)}\leq C_{s, P_{s+1}} \norm{h}_{\GG,s+1}.
\end{align}
if $r_0\leq s\leq m-1$ for any $X^\chi_j$ with $\chi\in \Delta''(b)$ and $\chi<1$.

Above estimates and Theorem \ref{th:4} imply that $X^\mu_iD_v\in \mathcal{H}^{m-2}$ of $\pi$ with estimates:
 \begin{align}\label{for:55}
\norm{X^\mu_if_v}_{\GG,s}\leq C_{s, P_1,P_{s+2}}\norm{h}_{\GG,s+2}.
\end{align}
for any $r_0\leq s\leq m-2$.

If $\mu<1$, by using the form $D_v^-$ we see that \eqref{for:23} also holds if $1\leq s\leq m-1$ for any $X^\chi_j$ with $\chi\in \Delta''(b)$ and $\chi<1$ or $\chi\in\Delta'(b)\backslash\Delta''(b)$. Furthermore, by using the form $D_v^+$ we see that \eqref{for:25} also holds if $1\leq s\leq m-1$ for any $X^\chi_j$ with $\chi\in \Delta''(b)$ and $\chi>1$. Hence, \eqref{for:55} also follows for the case of $\mu<1$.

Now we are ready to show the differentiability of $f_v$ in other directions of $\mathbf{g}$.
Let $Y_j=p_j(X^{\mu_{l(1)}}_{j(1)},\cdots,X^{\mu_{l(i)}}_{j(i)})$, $1\leq j\leq
r(a)$ (see \eqref{for:4}). We note that
\begin{align*}
  Y_j^sf_v= Y_j^{s-1}p_j(X^{\mu_{l(1)}}_{j(1)},\cdots,X^{\mu_{l(i)}}_{j(i)})f_v,
\end{align*}
and $\mu_{l(k)}\neq 1$ for any $1\leq k\leq i$. Then by using \eqref{for:55} we obtain:
\begin{align}\label{for:52}
 \norm{Y_j^sf_v}\leq C_{s,P_1,P_{s+m_0+1}}\norm{h}_{\GG,s+m_0+1}
\end{align}
for any $1\leq s\leq m-m_0-1$.

Using \eqref{for:37}, \eqref{for:22} and \eqref{for:52}  it follows from Theorem $3.1$ of \cite{Goodman1} that $f_v\in \mathcal{H}^{m-m_0-2}$ in $\pi'$; furthermore, Theorem \ref{th:4} shows that
\begin{align}\label{for:56}
  \norm{f_v}_{\GG',s}\leq C_{s,P_1,\cdots,P_{s+m_0+2}}\norm{h}_{\GG,s+m_0+2}+C_s\norm{f_v}.
\end{align}
If $1\leq s\leq m-m_0-2$.

Especially, let $s=1$ we have
\begin{align}\label{for:57}
  \norm{f_v}_{\GG',1}\leq C_{P_1,\cdots,P_{m_0+3}}\norm{h}_{\GG,m_0+3}+C\norm{f_v}.
\end{align}

It it fairly clear what one has to do now. We show how to get the estimate of $\norm{f_v}$. By using \eqref{for:54} and \eqref{for:21} we have
\begin{align*}
\norm{f_v}^2&=\big|
\sum_{j=0}^{+\infty}\lambda^{-(j+1)}\pi(b^j)v_j,f_v\rangle\big|\leq
\sum_{j=0}^{+\infty}\big|\langle \pi(b^j)v_j,f_v\rangle\big|\\
&\leq
C_{P_1}\norm{h}_{\GG,1}(\norm{f_v}+\norm{f_v}_{\GG',1}).
\end{align*}
Together with \eqref{for:57} we get
\begin{align*}
\norm{f_v}\leq C_{P_1,\cdots,P_{m_0+3}}\norm{h}_{\GG,m_0+3}.
\end{align*}
Combined with \eqref{for:56}, we obtain
\begin{align*}
  \norm{f_v}_{\GG',s}\leq C_{s,P_1,\cdots,P_{s+m_0+2}}\norm{h}_{\GG,s+m_0+2}+C_{s,P_1,\cdots,P_{m_0+3}}\norm{h}_{\GG,m_0+3}
\end{align*}
if $1\leq s\leq m-m_0-2$. Hence we finish the proof.

\end{proof}

We are now in a position to proceed with the proof of Theorem \ref{th:6}.

\subsection{Proof of Theorem \ref{th:6}} \emph{\textbf{Proof of \eqref{for:3}}}. It clear that $h\in \mathcal{H}^1$ by assumption. From \eqref{for:60} of Lemma \ref{le:8} we see that $\sum_{j=-\infty}^{+\infty}\lambda^{-(j+1)} \pi(a^j)h$ is a distribution in $\pi$. By using $h=\pi(a)f-\lambda f$ we have
\begin{align*}
  \sum_{j=-\infty}^{+\infty}\lambda^{-(j+1)} \pi(a^j)h&=\sum_{j=-\infty}^{+\infty}\lambda^{-(j+1)} \pi(a^j)\big(\pi(a)f-\lambda f\big)\\
  &\overset{\text{(1)}}{=}\sum_{j=-\infty}^{+\infty}\lambda^{-(j+1)} \pi(a^{j+1})f-\sum_{j=-\infty}^{+\infty}\lambda^{-j} \pi(a^j)f\\
  &=0,
\end{align*}
where $(1)$ follows from the fact that $\sum_{j=-\infty}^{+\infty}\lambda^{-(j+1)} \pi(a^j)f$ is also a distribution in $\pi$. Then we finish the proof.

\medskip
\noindent\emph{\textbf{Proof of \eqref{for:33}}}. The assumption that $\sum_{j=-\infty}^{+\infty}\lambda^{-(j+1)} \pi(b^j)h=0$ as a distribution in $\pi$ implies that $\sum_{j=-\infty}^{+\infty}\lambda^{-(j+1)} \pi(b^j)h=0$ as a distribution in $\pi'$ (since $\mathcal{H}^{\infty}(\pi)\subset \mathcal{H}^{\infty}(\pi')$). Then the equation $\rho(b)f-\lambda f=h$ has two distributional solutions
\begin{align*}
       f=D_h^{\binom{+}{-}}=\binom{-}{+}\sum_{\binom{j\geq0}{j\leq
-1}}\lambda^{-(j+1)}\pi(b^j)h
      \end{align*}
in $\pi'$ by \eqref{for:60} of Lemma \ref{le:8}.

From \eqref{for:61} of Lemma \ref{le:8} we see that $f\in\mathcal{H}^{m-m_0-3}$ of $\pi'$ with the following estimates
\begin{align}\label{for:59}
\norm{f}_{\GG',m-m_0-3}\leq C_{m} \norm{h}_{\GG,m}.
\end{align}
Next, we will show how to obtain differentiability along directions outside $\mathbf{g}$.  For any $X^\mu_i$, $\mu\in\Delta'''(b)$ (see \eqref{for:74}), $1\leq i\leq \dim\mathfrak{g}_\mu$ and $1\leq s\leq m-1$, set $Y_j=\textrm{Ad}(y_{b,j}^{-1})(X^\mu_i)$, $j\in\ZZ$. From \eqref{for:21}  we have
\begin{align*}
  \norm{(Y_j)^s h}_{\GG,t}\leq C(1+\epsilon)^{\abs{j}}(\abs{j}+1)^{s\dim \mathfrak{G}}\norm{h}_{\GG,s+t}.
\end{align*}
all $0\leq s+t\leq m-1$.

From \eqref{for:60} of Lemma \ref{le:8} we see that
\begin{align*}
 (X_i^\mu)^sf&=-\sum_{j=0}^\infty \lambda^{-1}(\lambda\mu^s)^{-j}\pi_1(b^j)((Y_j)^s h)=\sum_{j=-1}^{-\infty} \lambda^{-1}(\lambda\mu^s)^{-j}\pi_1(b^j)((Y_j)^s h)
 \end{align*}
are distributions in $\pi'$ for any $0\leq s\leq m-2$ if $\mu$ is sufficiently close to $1$, i.e., $b$ is sufficiently close to $a$. Furthermore, \eqref{for:61} of Lemma \ref{le:8} implies that $(X_i^\mu)^sf\in\mathcal{H}$ with the estimate
\begin{align*}
 \norm{(X_i^\mu)^sf}\leq C_{a,s}\norm{h}_{s+m_0+3}.
\end{align*}
for any $s\leq m-m_0-3$.

This together with \eqref{for:59} imply that $f\in\mathcal{H}^{m-m_0-3}$ with the estimate
\begin{align*}
 \norm{f}_{\GG,m-m_0-3}\leq C_{a,m}\norm{h}_{m}
\end{align*}
from Theorem \ref{th:4}. Hence we finish the proof.

\medskip
\noindent\emph{\textbf{Proof of \eqref{for:35}}}. We just need to show that if $\mathcal{D}(h)=0$ for any $(\rho(a)-\lambda I)$-invariant distribution $\mathcal{D}\in \mathcal{H}^{-1}$, then $\sum_{j=-\infty}^{+\infty}\lambda^{-(j+1)} \pi(a^j)h=0$ as a distribution. Note that for any $f\in\mathcal{H}^\infty$, $D_f:=\sum_{j=-\infty}^{+\infty}\bar{\lambda}^{-(j+1)}\pi(a^{-j})f\in \mathcal{H}^{-1}$ by \eqref{for:60} of Lemma \ref{le:8}. Furthermore, for any $g\in \mathcal{H}^1$
\begin{align*}
 &\sum_{j=-\infty}^{+\infty}\big\langle \pi(a)g-\lambda g,\bar{\lambda}^{-(j+1)}\pi(a^{-j})f\big\rangle\\
 &=
 \sum_{j=-\infty}^{+\infty}\big\langle \lambda^{-(j+1)}\pi(a^{j+1})g-\lambda^{-j}\pi(a^{j})g,f\big\rangle\\\
 &=0.
\end{align*}
This shows that $D_f\in \mathcal{H}^{-1}$ is $(\rho(a)-\lambda I)$-invariant. On the other hand, if $D_f(h)=0$ for any $f\in \mathcal{H}^\infty$, then
\begin{align*}
 &0=\sum_{j=-\infty}^{+\infty}\big\langle h,\bar{\lambda}^{-(j+1)}\pi(a^{-j})f\big\rangle=\sum_{j=-\infty}^{+\infty}\big\langle \lambda^{-(j+1)}\pi(a^{-(j+1)})h,f\big\rangle.
\end{align*}
This shows that $\sum_{j=-\infty}^{+\infty}\lambda^{-(j+1)} \pi(a^j)h=0$ as a distribution. This proves the result.

\subsection{Cohomological equation for partially hyperbolic flow} Recall that for a $C^1$ cocycle $\beta$, if $A=\RR^k$ the \emph{infinitesimal generator} of $\beta$ is defined by
\begin{align}\label{for:72}
\vartheta(\nu,\beta)=\frac{d}{dt}\beta(\exp t\nu)\bigl|_{t=0}
\end{align}
The cocycle identity and commutativity of $A$ imply that $\vartheta$ is a closed $1$-form on
the $A$-orbits in $X$. We can also recover $\beta$ from $\vartheta$ by
\begin{align*}
    \beta(\exp X)=\int_{0}^1\vartheta(X,\beta)\cdot \exp tXdt
\end{align*}
Thus, if $A=\RR^k$ we can restrict our attention to infinitesimal version of the cohomology equation
$\vartheta=\eta-dH$, where $\eta$ is another infinitesimal generator of a smooth cocycle and $H$ is the transfer function. Therefore
a cocycle $\beta$ is cohomologous to a constant cocycle if the associated $1$-form $\vartheta$ is exact and the problem of finding which
cocycle is cohomologous to a  trivial one boils down to the problem of determining
which closed $1$-form on the orbit foliation is exact. In fact, this point of view is
the most useful for our purposes.

For the cohomological equation $\mathfrak{v}f=h$ where $\mathfrak{v}\in\mathfrak{G}$ is a partially hyperbolic element, i.e., the spectrum of Ad$(\exp(t\mathfrak{v}))$ on $\mathfrak{G}$ is not contained in $U(1)$, we get results similar to the discrete-action cases.
\begin{corollary}\label{cor:3}
Suppose $(\pi,\,\mathcal{H})$ is a unitary representation  of $\GG$ such that the restriction of $\pi$ to any
simple factor of $G$ is isolated from the trivial representation (in the Fell topology) if $\GG=G$; or $\pi$ contains no non-trivial $V$-fixed vectors if $\GG=G\ltimes_\rho V$.  If $\mathfrak{v}\in\mathfrak{G}$ is partially hyperbolic, then:
\begin{enumerate}
\item if $f\in \mathcal{H}^1$ is a solution of the equation: $\mathfrak{v}f=h$, then
\begin{align*}
 \int_{-\infty}^\infty\pi(\exp(t\mathfrak{v}))hdt=0
\end{align*}
as a distribution.

\medskip
  \item  there exist constants $m_0>\sigma>1$ (only depending on $\GG$) such that for any $m\geq m_0$, exist $\delta(m)>0$, such that for any $\mathfrak{u}\in \mathfrak{G}$ with $\norm{\mathfrak{v}-\mathfrak{u}}<\delta$, if $h\in \mathcal{H}^m$ satisfying $\int_{-\infty}^\infty\pi(\exp(t\mathfrak{u}))hdt=0$ as a distribution, then the equation $\mathfrak{u} f=h$
have a solution $f\in\mathcal{H}^{m-\sigma}$ and the following
estimate holds
\begin{align*}
\norm{f}_{m-\sigma}\leq C_{m} \norm{h}_{m}.
\end{align*}

\medskip
 \item if $h\in H^\infty$ and $\mathcal{D}(h)=0$ for any $\mathfrak{v}$-invariant distribution $\mathcal{D}\in \mathcal{H}^{-1}$, then the cohomological equation $\mathfrak{v}f=h$, has a solution $f\in \mathcal{H}^\infty$.

\end{enumerate}

\end{corollary}
\begin{proof}
If $\GG=G$,  for any $\mathfrak{u}\in \mathfrak{G}$, it has the
Iwasawa decomposition
\begin{align*}
\mathfrak{u}=k_{\mathfrak{u}}+ x_{\mathfrak{u}}+n_{\mathfrak{u}}
\end{align*}
for $3$ commuting elements, where $k_{\mathfrak{u}}$ is compact, $x_{\mathfrak{u}}$ is in a $\RR$-split Cartan algebra, $n_{\mathfrak{u}}$ is nilpotent.

For partially hyperbolic $\mathfrak{u}$, $x_{\mathfrak{u}}$ is non-trivial. If $\GG=G\ltimes_\rho V$ for any $\mathfrak{u}\in \mathfrak{G}$, we have the decomposition
$\mathfrak{u}=g_{\mathfrak{u}}+v_{\mathfrak{u}}$, where $g_{\mathfrak{u}}\in\text{Lie}(G)$ and $v_{\mathfrak{u}}\in V$. For partially hyperbolic $\mathfrak{u}$,  $x_{g_{\mathfrak{u}}}$ is non-trivial.

For partially hyperbolic $\mathfrak{u}$, similar to \eqref{for:5} we consider the eigenspace decomposition of
$\mathfrak{G}$ for $\text{Ad}(\exp(x_{\mathfrak{u}}))$ or $\text{Ad}(\exp(x_{g_{\mathfrak{u}}}))$:
\begin{align*}
\mathfrak{G}=\sum_{\mu\in\Delta(\mathfrak{u})}\mathfrak{g}_\mu.
\end{align*}
If $\GG=G\ltimes_\rho V$, for $\mathfrak{u}=g_{\mathfrak{u}}+v_{\mathfrak{u}}$, we have a decomposition of $v_{\mathfrak{u}}=\sum_{\mu\in \Delta(a)} v_{\mathfrak{u},\mu}$, where $v_{\mathfrak{u},\mu}\in\mathfrak{g}_\mu\bigcap V$. Set $u=\sum_{1\neq\mu\in\Delta(\mathfrak{u})}v_{\mathfrak{u},\mu}$.
We note that $d\rho(g_{\mathfrak{u}})$ restricted on the subspace $W=\sum_{1\neq\mu\in \Delta(a)} \mathfrak{g}_\mu\bigcap V$ is invertible. Set $v=(d\rho(g_{\mathfrak{u}}))^{-1}(\sum_{1\neq\mu\in\Delta(\mathfrak{u})}v_{\mathfrak{u},\mu})$. By using \eqref{for:19} we get
\begin{align*}
  v\exp(\mathfrak{u})v^{-1}=\big(g_{\mathfrak{u}},v_{\mathfrak{u},1}\big).
\end{align*}
Hence we can just assume that $v_{\mathfrak{u}}\in \mathfrak{g}_1$. Then we can write
\begin{align*}
 \mathfrak{u}=x_{\mathfrak{u}}+(k_{\mathfrak{u}}+n_{\mathfrak{u}}+v_{\mathfrak{u}})
\end{align*}
as $2$ commuting elements. Set $y_{\mathfrak{u}}=k_{\mathfrak{u}}+n_{\mathfrak{u}}$ if $\GG=G$ or $y_{\mathfrak{u}}=k_{\mathfrak{u}}+n_{\mathfrak{u}}+v_{\mathfrak{u}}$ if $\GG=G\ltimes_\rho V$. Then we have
\begin{align*}
  \exp(t\mathfrak{u})=\exp(tx_{\mathfrak{u}})\cdot \exp(ty_{\mathfrak{u}}),\qquad \forall t\in\RR.
\end{align*}
We note that the formal solutions to the equation $\mathfrak{u} f=h$ are:
\begin{align*}
       f^{\binom{+}{-}}=\binom{-}{+}\int_{\binom{t\geq0}{t\leq0}}\pi(\exp(t\mathfrak{u}))hdt.
      \end{align*}
Then we can follow the proof scheme of Lemma \ref{le:8} and Theorem \ref{th:6} to obtain the results.

\end{proof}

\section{Twisted cocycle rigidity}
\subsection{Higher rank trick and trivialization of cohomology}
Now  we will  show that  in the  higher rank case  obstructions  to solving the cocycle equation:
\begin{align}\label{for:62}
 \pi(a_1)f-\lambda_1 f=\pi(b_1)h-\lambda_2 h
\end{align}
vanish. Here we assume that $\lambda_1,\,\lambda_2\in U(1)$ and $a_1$ and $b_1$ are commuting partially hyperbolic elements.
The reason for that  is the commutation relation
\eqref{for:62} means that  the pair $f,\,h$
form a twisted cocycle over the homogeneous action generated by $a_1$ and $b_1$. Joint  solvability of the  cocycle equations for commuting elements  means that this  cocycles is  a coboundary, hence corresponding  twisted first cohomology is trivial. The key ingredient in the proof is the ``higher rank trick'' that proves vanishing of the obstructions in both equations:
\begin{align*}
&\pi(a_1) p-\lambda_1p=h \notag\\
&\pi(b_1)p-\lambda_2p=f
\end{align*}
if the pair pair $f,\,h$ satisfy condition \eqref{for:62}. It
appears in virtually identical form in all proofs of cocycle and differentiable rigidity for actions of higher rank abelian groups that  use some form  of dual, i.e. harmonic analysis arguments.
For its earliest appearance see  Lemmas 4.3,  4.6 and 4.7 in \cite{Spatzier1}. Recall notations in Section \ref{sec:2}.
\begin{lemma}\label{le:9}
If $ab=ba$ in $\GG$, then:
\begin{enumerate}
  \item \label{for:43}if $\GG=G$, then $a^mb^j=x_a^mx_b^jk_a^mn_a^mk_b^jn_b^j$, for any $m,\,j\in\ZZ$.

  \medskip
  \item \label{for:44} if $\GG=G\ltimes_\rho V$ and $v_a$ is in the $1$-weight space for $\text{Ad}(x_{g_a})$, then $v_b$ is also in the $1$-weight space for $\text{Ad}(x_{g_b})$; and
      \begin{align*}
       a^mb^j=x_{g_a}^mx_{g_b}^j(k_{g_a}n_{g_a}, v_a)^m(k_{g_b}n_{g_b}, v_b)^n,\qquad \forall m,\,j\in\ZZ.
      \end{align*}

      \medskip
\item \label{for:45} suppose $x_a$ and $x_b$ (resp. $x_{g_a}$ and $x_{g_b}$) are linear independent and $\lambda_1,\,\lambda_2\in U(1)$.  Also suppose $h=\{h_{n,j}\}$, $n,\,j\in\ZZ$ is a sequence in $\mathcal{H}^1(\pi)$ satisfying $\norm{h_{n,j}}_{\GG,1}\leq P(\abs{j},\,\abs{n})$, where $P$ a polynomial, then
\begin{align}\label{for:40}
  \sum_{n=-\infty}^{\infty}\sum_{j=-\infty}^{\infty}\lambda_1^{-(j+1)}\lambda_2^{-(n+1)}\pi(b^na^j)h_{n,j}
\end{align}
is a
distribution
      \end{enumerate}
\end{lemma}
\begin{proof}
\noindent\emph{\textbf{Proof of \eqref{for:43}}}: Let $c$ stand for $a$ or $b$. The adjoint map $\text{Ad}(c)$ decomposes as
\begin{align*}
 \text{Ad}(c)=\text{Ad}(k_c)\text{Ad}(x_c)\text{Ad}(n_c),
\end{align*}
where $\text{Ad}(k_c)$ is compact, $\text{Ad}(x_c)$ is in a $\RR$-semisimple, $\text{Ad}(n_c)$ is unipotent. We claim that $\text{Ad}(x_a)$ commute with $\text{Ad}(k_b)$, $\text{Ad}(x_b)$ and $\text{Ad}(n_b)$. Indeed,
we can consider the decomposition of $\mathfrak{G}$ into generalized
eigenspaces for the action of $\text{Ad}(a)$. By assumption
$\text{Ad}(x_a)$ is a scalar multiple of identity on each generalized eigenspace of $\text{Ad}(a)$.
Since $b$ commutes with $a$, $\text{Ad}(x_b)$ preserves each generalized eigenspace of $\text{Ad}(a)$
and acts by a scalar multiple in each, which implies the claim.

It is clear that $x_an_b=n_bx_a$, $x_ax_b=x_bx_a$ and $x_ak_b=k_bx_az$, where $z\in Z(G)$, the center of $G$. Their exists $l\in\NN$ such that
$k_bx_a^lk_b^{-1}=x_a^l$ by noting that $Z(G)$ is finite. This shows that $\text{Ad}(k_b)\log x_a=\log x_a$. Then it follows that
$x_a$ and $k_b$ commute. Then we finish the proof.

\medskip
\noindent\emph{\textbf{Proof of \eqref{for:44}}}: By Proposition  $4.4$ of \cite{Zhenqi0} there is $s\in\GG$ such that
\begin{align*}
 scs^{-1}=(g_c,v_c'),\qquad c=a,\text{ or }b
\end{align*}
where $v_c'$ is in the $1$-weight spaces both for $\text{Ad}(x_{g_a})$ and $\text{Ad}(x_{g_b})$. This and previous result imply the conclusion immediately.

\medskip
\noindent\emph{\textbf{Proof of \eqref{for:45}}}: If $\GG=G$ it follows from \eqref{for:43} that
\begin{align*}
 d(b^na^j,e)\geq C\sum_{\mu\in\Delta(b^na^j)}\abs{\log\mu},\qquad \forall j,\,n\in\ZZ.
\end{align*}
For any $j,n\in\ZZ$ we note that
\begin{align*}
  \sum_{\mu\in\Delta(b^na^j)}\abs{\log\mu}=\sum_{\substack{\mu_l\in\Delta(a),\\\lambda_l\in\Delta(b)}}\abs{j\log\mu_l+n\log\lambda_l}
\end{align*}
where $\mu_l$ and $\lambda_l$ are eigenvalues of $\text{Ad}(x_{a})$ and $\text{Ad}(x_{b})$ on $\mathfrak{G}$ respectively for the same eigenvectors. Write
\begin{align*}
\sum_{\substack{\mu_l\in\Delta(a),\\\lambda_l\in\Delta(b)}}\abs{j\log\mu_l+n\log\lambda_l}
=(\abs{j}+\abs{n})\sum_{\substack{\mu_l\in\Delta(a),\\\lambda_l\in\Delta(b)}}\abs{j_1\log\mu_l+j_2\log\lambda_l}
\end{align*}
where $j_1=\frac{j}{\abs{j}+\abs{n}}$ and
$j_2=\frac{n}{\abs{j}+\abs{n}}$.

Since $x_a$ and $x_b$ are linearly  independent elements,
\begin{align*}
  c_0=\min_{\substack{\abs{r_1}+\abs{r_2}=1\\(r_1,r_2)\in\RR^2}}\sum_{\substack{\mu_l\in\Delta(a),\\\lambda_l\in\Delta(b)}}\frac{1}{2}
  \abs{r_1\log\mu_l+r_2\log\lambda_l}>0.
\end{align*}
Hence we have
\begin{align*}
  d(b^na^j,e)\geq Cc_0(\abs{j}+\abs{n}).
\end{align*}
If $\GG=G$, for any $u\in \mathcal{H}^\infty$ by Theorem \ref{th:2}, we have
\begin{align*}
&\sum_{n=-\infty}^{\infty}\sum_{j=-\infty}^{\infty}\big|\langle
\lambda_1^{-(j+1)}\lambda_2^{-(n+1)}\pi(b^na^j)h_{n,j},u\rangle\big|\\
&=\sum_{n=-\infty}^{\infty}\sum_{j=-\infty}^{\infty}\abs{\langle
\pi(b^na^j)h_{n,j},u\rangle}\\
&\leq
\sum_{n=-\infty}^{\infty}\sum_{j=-\infty}^{\infty}C\norm{h_{n,j}}_{1}\norm{u}_{1}e^{-\gamma
c_0 (\abs{n}+\abs{j})}\\
&\leq
\sum_{n=-\infty}^{\infty}\sum_{j=-\infty}^{\infty}CP(\abs{j},\,\abs{n})\norm{u}_{1}e^{-\gamma
c_0 (\abs{n}+\abs{j})}<\infty,
\end{align*}
where $\gamma$ is a constant only dependent on $G$. This shows that \eqref{for:40} is a distribution.

If $\GG=G\ltimes_\rho V$, denote by $y_c=(k_{g_c}n_{g_c},v_c)$; and for simplicity, denote $x_{g_c}$ by $x_c$, where $c$ stands for $a$ or $b$. By  arguments at the beginning of the proof of Lemma \ref{le:8} and Theorem \ref{th:2} for any $u\in\mathcal{H}^\infty$ we have
\begin{align*}
\abs{\langle\pi(b^na^j)h,u\rangle}&\overset{\text{(1)}}{=}\big|\langle\pi(x_{b}^nx_{a}^j)(\pi(y_{b}^ny_{a}^j)h_{n,j}),u\rangle\big|\notag\\
&\leq C\big(\norm{\pi(y_{b}^ny_{a}^j)h_{n,j}}\norm{u}+\norm{u}\norm{\pi(y_{b}^ny_{a}^j)h_{n,j}}_{1}\\
&+\norm{\pi(y_{b}^ny_{a}^j)h_{n,j}}\norm{u}_{1}+\norm{u}_{1}\norm{\pi(y_{b}^ny_{a}^j)h_{n,j}}_{1}\big)e^{-\gamma_1d(x_{b}^nx_{a}^j,e)}\\
&=C\big(\norm{h_{n,j}}\norm{u}+\norm{u}\norm{\pi(y_{b}^ny_{a}^j)h_{n,j}}_{1}\\
&+\norm{h}\norm{u}_{1}+\norm{u}_{1}\norm{\pi(y_{b}^ny_{a}^j)h_{n,j}}_{1}\big)e^{-\gamma d(x_{b}^nx_{a}^j,e)},
\end{align*}
where $\gamma$ is a constant only dependent on $G$. Here in $(1)$ we used \eqref{for:44} of Lemma \ref{le:9}.

By using \eqref{for:19} it follows that for any $Y\in \mathfrak{G}$ we have
\begin{align*}
  \norm{\text{Ad}(y_{b}^ny_{a}^j)Y}\leq C(\abs{j}+1)^{\dim \mathfrak{G}}(\abs{n}+1)^{\dim \mathfrak{G}}\norm{Y},\qquad \forall j\in\ZZ.
\end{align*}
Also note that for any $Y\in \mathfrak{G}$
\begin{align*}
  Y\pi(y_{b}^ny_{a}^j)h_{n,j}=\pi(y_{b}^ny_{a}^j)\big(\text{Ad}(y_{a}^{-j}y_{b}^{-n})Y\big)h_{n,j}.
\end{align*}
Hence we get
\begin{align*}
\norm{\pi(y_{b}^ny_{a}^j)h_{n,j}}_{1}\leq C(\abs{j}+1)^{\dim \mathfrak{G}}(\abs{n}+1)^{\dim \mathfrak{G}}\norm{h_{n,j}}_{1}.
\end{align*}
Then the estimate for $\abs{\langle\pi(a^j)v,u\rangle}$ follows directly from \eqref{for:54}, \eqref{for:50} and \eqref{for:41}.

By arguments similar to the case of $\GG=G$, there is $c>0$ such that
\begin{align*}
 d(x_{b}^nx_{a}^j,e)\geq c(\abs{j}+\abs{n}).
\end{align*}
Then it follows that for any $n,\,j\in\ZZ$
\begin{align*}
\abs{\langle\pi(b^na^j)h,u\rangle}&\leq C(\abs{j}+1)^{\dim \mathfrak{G}}(\abs{n}+1)^{\dim \mathfrak{G}}\big(\norm{h_{n,j}}\norm{u}+\norm{u}\norm{h_{n,j}}_{1}\\
&+\norm{h_{n,j}}\norm{u}_{1}+\norm{u}_{1}\norm {h}_{1}\big)e^{-c\gamma (\abs{j}+\abs{n})}\\
&\leq C(\abs{j}+1)^{\dim \mathfrak{G}}(\abs{n}+1)^{\dim \mathfrak{G}}P(\abs{j},\,\abs{n})\norm{u}_1e^{-c\gamma (\abs{j}+\abs{n})}
\end{align*}
The above estimates imply that \eqref{for:40} is a distribution. Hence we finish the proof.

\end{proof}

To prove Theorem \ref{th:9}, we also need the following result:
\begin{fact}\label{fact:1}
 If $b\in\GG$ is partially hyperbolic and $\lambda\in U(1)$ and the equation $\pi(b)f-\lambda f=0$ has a solution $f\in \mathcal{H}^1(\pi)$ then $f=0$.
\end{fact}
\begin{proof}
 Since $\pi(b^j)f=\lambda^{j}f$ for any $j\in\ZZ$, by using estimates obtained in \eqref{for:64} of Lemma \ref{le:8}, for any $u\in \mathcal{H}^1(\pi)$ we get
\begin{align*}
    \langle f,\,u\rangle=\lim_{j\rightarrow\infty}\langle \lambda^{-j}\pi(b^j)f,\,u\rangle=0
\end{align*}
Then $f=0$ follows immediately.
\end{proof}
\subsection{Proof of Theorem \ref{th:9}} From the equation \eqref{for:62} we get
\begin{align*}
&\sum_{j=-n}^{j=n}\lambda_1^{-(j+1)}\pi(b_1a_1^j)h-\sum_{j=-n}^{j=n}\lambda_2\lambda_1^{-(j+1)}\pi(a_1^j)h\\
&=\lambda_1^{-(n+1)}\pi(a_1^{n+1})f-\lambda_1^{n}\pi(a_1^{-n})f.
\end{align*}
By \eqref{for:64} of Lemma \ref{le:8}, the right-hand converges to  $0$  as a distribution (in $\pi$) when
$n\rightarrow \infty$. Hence  we see that
\begin{align}\label{for:40}
\lambda_2^{-1}\sum_{j=-\infty
}^{+\infty}\lambda_1^{-(j+1)}\pi(b_1a_1^j)h=\sum_{j=-\infty
}^{+\infty}\lambda_1^{-(j+1)}\pi(a_1^j)h
\end{align}
as distributions. This shows that
\begin{align*}
\lambda_2^{-k}\sum_{j=-\infty
}^{+\infty}\lambda_1^{-(j+1)}\pi(b_1^ka_1^j)h=\sum_{j=-\infty
}^{+\infty}\lambda_1^{-(j+1)}\pi(a_1^j)h
\end{align*}
as distributions for any $k\in\ZZ$.

 On the other hand, by iterating equation \eqref{for:40}
we obtain:
\begin{align*}
&\sum_{n=-\infty
}^{+\infty}\sum_{j=-\infty
}^{+\infty}\lambda_2^{-(n+1)}\lambda_1^{-(j+1)}\pi(b_1^na_1^j)h=\sum_{n=-\infty
}^{+\infty}\lambda_2^{-1}\sum_{j=-\infty
}^{+\infty}\lambda_2^{-n}\lambda_1^{-(j+1)}\pi(b_1^na_1^j)h\notag\\
&=
\sum_{n=-\infty
}^{+\infty}\lambda_2^{-1}\sum_{j=-\infty
}^{+\infty}\lambda_1^{-(j+1)}\pi(a_1^j)h
\end{align*}
Since the series in the left hand side of above equation is a distribution by \eqref{for:45} of Lemma \ref{le:9}, it forces
$\sum_{j=-\infty
}^{+\infty}\lambda_1^{-(j+1)}\pi(a_1^j)h$ to be a $0$
distribution.  This is   the ``higher rank trick''!

By \eqref{for:33} of Theorem \ref{th:6} each equation of
\begin{align}\label{eq:7}
&\pi(a_1)p-\lambda_1p=h \notag\\
&\pi(b_1)p-\lambda_2p=f
\end{align}
has a $\mathcal{H}^{m-\sigma}$ solution. Moreover, we will show that they coincide. In the following proof, to simply notation, for any $g\in \GG$ and $\lambda\in \CC$, we define the linear operator $F(g,\lambda)$ on $\mathcal{H}$:
\begin{align*}
  F(g,\lambda)v=\pi(g)v-\lambda v, \qquad \forall \,v\in\mathcal{H}.
\end{align*}
If $p$ solves the first equation, i.e.
$F(a_1,\lambda_1)p=h$ then by equation \eqref{for:62} we
have
\begin{align*}
F(b_1,\lambda_2)\circ F(a_1,\lambda_1)p=F(b_1,\lambda_2)h=F(a_1,\lambda_1)f
\end{align*}
Since operators $F(a_1,\lambda_1)$ and $F(b_1,\lambda_2)$
commute this implies
\begin{align*}
F(a_1,\lambda_1)\big(F(b_1,\lambda_2)p-f\big)=0
\end{align*}
By Fact \ref{fact:1} $F(b_1,\lambda_2)p=f$ by noting that $p\in \mathcal{H}^1$, which is implied by assumption  $m\geq m_0+\sigma$ and $m_0>1$. Therefore $\pi(b_1)p-\lambda_2p=f$ i.e.,  $p$ solves the second equation as well.

\subsection{Cocycle rigidity for partially hyperbolic flow} Similar to Fact \ref{fact:1}, we have the following result for partially hyperbolic flow:
\begin{fact}\label{fact:2}
 If $X\in \mathfrak{G}$ is partially hyperbolic and the equation $Xf=0$ has a solution $f\in \mathcal{H}^1(\pi)$ then $f=0$.
\end{fact}
\begin{proof}
The assumption $Xf=0$ implies that $f$ is invariant under $\pi(\exp(tX))$, i.e., $\pi(\exp(tX))f=f$ for any $t\in\RR$. Then the result follows directly from Fact \ref{fact:1} immediately.
\end{proof}
Then by arguments in the proof of Corollary \ref{cor:3}, we can follow the proof line of Lemma \ref{le:9} and Theorem \ref{th:9} to obtain the cocycle rigidity for $\RR^2$ partially hyperbolic actions.

Suppose $X,\,Y\in\mathfrak{G}$ are commuting partially hyperbolic and linearly independent and $m_0,\,\sigma$ as in Theorem \ref{th:6}.
\begin{corollary}\label{cor:4}
Suppose $(\pi,\,\mathcal{H})$ is a unitary representation  of $\GG$ such that the restriction of $\pi$ to any
simple factor of $G$ is isolated from the trivial representation  if $\GG=G$; or $\pi$ contains no non-trivial $V$-fixed vectors if $\GG=G\ltimes_\rho V$. For any $m\geq m_0+\sigma$, exist $\delta(m)>0$, such that for any $X_1,\,Y_1\in \mathfrak{G}$ with $\norm{X-X_1}+\norm{Y-Y_1}<\delta$ and $X_1Y_1=Y_1X_1$, if $f,\,h\in \mathcal{H}^m$
and satisfy the cocycle equation
\begin{align*}
  X_1f=Y_1 h,
\end{align*}
then the equations
\begin{align*}
  Y_1p=f,\qquad X_1 p=h
\end{align*}
have a common solution $p\in \mathcal{H}^{m-\sigma}$ satisfying
the Sobolev estimate
\begin{align*}
    \norm{p}_{m-\sigma}\leq C_{m,X,Y}\max\{\norm{h}_{m}, \,\norm{f}_{m}\}.
\end{align*}
\end{corollary}

\section{Application to algebraic partially hyperbolic actions }

Recall that  $G$ denotes a real semisimple connected Lie group  of $\RR$-rank $\ge 2$ without compact factors and with finite center and $\Gamma$ a torsion free irreducible lattice in $G$.
\begin{definition} {\em Coarse Lyapunov distributions} are defined as  minimal
non-trivial intersections of stable distributions of various action
elements.\end{definition}

In the setting of present paper those are homogeneous distributions
or their perturbations,  that integrate to homogeneous foliations
called {\em coarse Lyapunov foliations}; see \cite[Section
2]{Damjanovic1} and \cite{Kalinin} for detailed discussion in
greater generality.

\subsection{Symmetric space examples}  For any abelian set $A \subset G$ there exists a $\RR$-split Cartan subgroup $A_0$ such that $x_a\in A_0$ (see \eqref{for:9}) for any $a\in A $ (see Proposition 4.2 of \cite{Zhenqi0}). Set $A'=\{x_a,\,a\in A\}$. Then $A'$ is a subgroup of $A_0$. We consider the
decomposition of $\mathfrak{g}$ with respect to the adjoint
representation of $A'$ and the resulting root system $\Delta_A$ is
called the {\em restricted root system with respect to $A$}. Then we get the decomposition of the Lie algebra
$\mathfrak{g}$ of $G$:
\begin{align}\label{for:6}
\mathfrak{g}=\mathfrak{g}_0 + \sum_{\mu\in\Delta_A}\mathfrak{g}_\mu
\end{align}
where $\mathfrak{g}_\mu$ is the root space of $\mu$ and
$\mathfrak{g}_0$ is the Lie algebra of the centralizer $C_G(p(A))$ of $A'$.

Elements of $\{a\in A:\,x_a\notin\bigcup_{\phi\in\Delta_A}\ker(\phi)\}$ are
\emph{regular} elements for $\alpha_A$. Connected components of the set of regular
elements are \emph{Weyl chambers}. For any $\mu \in \Delta_A$ let $\mathfrak{g}^{(\mu)}=\sum_{k>0}\mathfrak{g}_{k\mu}$ and
$U_{[\mu]}$ be the corresponding subgroup of $G$. Then these subalgebra $\mathfrak{g}^{(\mu)}$ form coarse Lyapunov distributions and (double) cosets of these subgroups $U_{[\mu]}$ form coarse Lyapunov foliations of $\alpha_A$, which coincide with those of $\alpha_{p(A)}$ (see Proposition 4.2 of \cite{Zhenqi0}).

If $A'=A_0$, the left translations  of $A$ on $ G/\Gamma$  is sometimes referred to as  {\em full Cartan action} (see \cite{Spatzier1}). If the coarse Lyapunov foliations of $\alpha_A$ coincide with those of $\alpha_{A_0}$, then $A'$ is in a generic position (see \cite{Damjanovic1}) and the action of $A$ on $ G/\Gamma$ is called a \emph{generic restriction}.

 \begin{remark}If $A'=A_0$, i.e., $\Delta_{A}$ is the standard root system. Calling $\Delta_{A}$ a restricted root system is somewhat abusive. Indeed, $\Delta_{A}$ does not carry the usual structures of a (reduced) root system, such as a canonical inner product and associated Weyl group. For more details, see Section 1.1 of \cite{Zhenqi0}.
 \end{remark}

\subsection{Twisted symmetric space examples}
Let $\rho:\Gamma\rightarrow
SL(N,\ZZ)$ be a representation of $\Gamma$ which admits no invariant
subspace on which $\Gamma$ acts trivially. We also assume that the Zariski closure of $\rho(\Gamma)$ has no compact factors.

 Then $\Gamma$ acts on the $N$-torus
$\mathbb T^N$ via $\rho$ and hence on $G\times\TT^N$ via
\begin{align*}
 \gamma(g,t)=(g\gamma^{-1},\,\rho(\gamma)t).
\end{align*}
Let $M=G\times\TT^N/\Gamma$ be the quotient of this action. $A$ acts on the product $G\times\TT^N$ given by $ a(g,t)=(ag,t)$ and since, the action of $A$
and $\Gamma$ commute it induces an
action of $A$ on $M$, which is the suspension of $\Gamma$-action on $\TT^N$.

We can assume $G$ has the following property: every Lie algebra homomorphism $\mathfrak{g}\to \mathfrak{sl}(N,\RR)$ is the derivative of a Lie group homomorphism $G\to SL(N,\RR)$. Otherwise we pass to some finite cover of $G$.

By Margulis' superrigidity theorem
\cite{Margulis}, semisimplicity of the algebraic hull $H$ of
$\rho(\Gamma)$ and the non-compactness of $\rho(\Gamma)$ the
representation $\rho$ of $\Gamma$ extends to a rational homomorphism
$G\rightarrow H_{ad}$ over $\RR$ where $H_{ad}$ is the adjoint group
of $H$. Note that $\rho(\Gamma)$ has finite center $Z$ (which follows,
eg, from Margulis' finiteness theorem \cite{Margulis}), then
$G$ acts on the orbifold $\RR^N/Z$ via $\rho$, which can be lifted to a
representation of $G$ on $\RR^N$, which we denote by
$\hat{\rho}$. Then $\rho(\gamma)\hat{\rho}(\gamma)^{-1}\in Z$ for any $\gamma\in \Gamma$. Then, by passing to a finite index subgroup $\Gamma_1$ of $\Gamma$, we get $\rho(\gamma)=\hat{\rho}(\gamma)$ for any $\gamma\in \Gamma_1$. Then $G\times\TT^N/\Gamma_1$ is a finite cover of $M$.
For simplicity, we still use $\rho$ to denote the lifted representation of $G$ on $\RR^N$.

We can build the associated semi-direct product $G_{\rho}=G\ltimes_\rho \RR^N$. The
multiplication of elements in $G_\rho$ is given by \eqref{for:10}, which shows that $G\times\TT^N/\Gamma_1=G_\rho/\Gamma_1\ltimes\ZZ^N$. The $A$ action on $G\times\TT^N/\Gamma_1$ is isomorphic to the action of $A$ as a homogeneous flow on $G_\rho/\Gamma_1\ltimes\ZZ^N$. We can
view $G\times\TT^N/\Gamma_1$ as a torus bundle over
$G/\Gamma_1$.
\begin{remark}
Firstly, if the Zariski closure of $\rho(\Gamma)$ has compact factors, we can pass to a suspension space (see \cite{Zhenqi111}).

Secondly, passing to finite covers of the homogeneous actions will not affect the local rigidity results. Indeed, the constriction of the transfer map relies on vanishing of the obstructions (see Section \ref{sec:4}). If the obstructions vanishes on a finite cover, then it also vanishes on the original space  (see Lemma \ref{le:16}). This allows us to assume that $\Gamma_1=\Gamma$. We denote by $\Gamma_\rho=\Gamma\ltimes\ZZ^N$.
\end{remark}
For any abelian set $A \subset G_\rho$ and set $A'=\{x_{g_a}:\,a\in A\}$ (see below \eqref{for:10}), there exists an element $s\in G_\rho$ such that the coarse Lyapunov distributions for the action of $\alpha_{sAs^{-1}}$ is the same as those for $\alpha_{A'}$  (see Proposition 4.4 of \cite{Zhenqi0}).
Let $\Phi_{A}$ denote the \emph{restricted weights of $G$ with respect to $A'$}.  Then the Lie algebra
$\mathfrak{g}_\rho$ of $G_\rho$ decomposes
\begin{align*}
\mathfrak{g}_\rho=\mathfrak{g}_0 + \sum_{r \in\Delta_A}\mathfrak{g}_r +\sum_{\mu\in\Phi_{A}}\mathfrak{e}_\mu
\end{align*}
where $\mathfrak{e}_\mu$ is the weight space of $\mu$ and
$\mathfrak{g}_0$ is the Lie algebra of the centralizer $C_{G_\rho}(A')$ of $A'$ in $\mathfrak{g}_\rho$.

Note that if $A'=A_0$, $\Phi_{A}$ is the standard weights. For any $r\in \Delta_A\cup \Phi_{A}$ let $\mathfrak{g}^{(r)}=\sum_{k\in \RR_+} \left(\mathfrak{g}_{kr}+\mathfrak{e}_{kr}\right)$ and
$U_{[r]}$ be the corresponding subgroup of $G_\rho$. Note that $r$ or $\mu$ may appear in the set of both restricted roots and weights.  These subalgebra $\mathfrak{g}^{(r)}$ form coarse Lyapunov distributions and (double) cosets of these subgroups $U_{[r]}$ form coarse Lyapunov foliations of $\alpha_A$, which coincide with those of $\alpha_{A'}$ (see Proposition 4.2 of \cite{Zhenqi0}).

Elements of $\{a\in A:\,x_{g_a}\notin\bigcup_{\phi\in\Delta_A\cup\Phi_{A}\setminus \{0\}}\ker(\phi)\}$ are \emph{regular} for $\alpha_A$ and connected components of the set of regular
elements \emph{Weyl chambers}. If $A=A_0$, the left translations  of $A$ on $ G_\rho/\Gamma_\rho$  is sometimes referred to as  {\em full Cartan action}.

Similarly to the symmetric space setting we will consider actions of
higher rank subgroups of $A$ by left
translations on double coset space $L\backslash G_\rho/\Gamma_\rho$ where $L$ is a compact subgroup commuting with $A$.
\subsection{Higher rank restrictions and standard perturbations}\label{regularrestrictions}
Let $X$  be   a double coset space $L\backslash G/\Gamma$ as in
symmetric space examples or  $L\backslash
G_\rho/\Gamma_\rho$ as in twisted symmetric space
examples; and let $\bar{X}$  be   a coset space $G/\Gamma$ as in
symmetric space examples or  $G_\rho/\Gamma_\rho$ as in twisted symmetric space
examples. We consider the action $\alpha_A$ on both $X$ and $\bar{X}$.

Since $A$ is the image of an embedding $i_0 :
\ZZ^k\times\RR^{\ell}\rightarrow A$, one can naturally consider
the  action $\alpha_{i_0}$  of $A=\ZZ^k\times\RR^l$  on $X$ (or on $\bar{X}$)
 given by
 \begin{align}\label{for:15}
 \alpha_{i_0}(a,x)=i_0(a)\cdot x
 \end{align}
Then we will say that $A$ action  $\alpha_{i_0}$ generates $A$ action $\alpha_{A}$ since $\alpha_{i_0}$ is $\alpha_{A}$ with a fixed system of coordinates.
Note that $A$ can be obtained as  the image of different embeddings; corresponding actions of $A$ differ by a time change. It is immediately obvious that if $\alpha_{i_0}$ is cocycle rigid then the same is true for any time change obtained by an automorphism of $A$; hence the notion of cocycle rigidity for $\alpha_{A}$  depends only on the subgroup  $A$.
\begin{definition}
$\alpha_A$ is called a \emph{higher-rank} partially action, if the set of coarse Lyapunov distributions of $\alpha_A$ is not generated by a rank-one subgroup.
\end{definition}
 It is clear that if $\alpha_A$ is higher-rank, then we can choose regular elements $a$ and $b$ of $A$ such that $x_a$ and $x_b$ (resp. $x_{g_a}$ and $x_{g_b}$) are linearly independent. $a$ and $b$ will be
referred to as {\em regular generators}.

Let $\mathfrak L$ be the Lie algebra of the group  $L$. Let $\mathcal{N}$ and $\mathfrak{N}$ denote the neutral distributions of $\alpha_A$ on $X$ and on $\bar{X}$
respectively (neutral distribution is the subspace spanned by Lyapunov distributions with $0$ Lyapunov exponents).  Then:

\begin{itemize}

\item
For the symmetric space examples $\mathcal{N}$ is $\mathfrak L\backslash\mathfrak{N}$
where $\mathfrak{N}=\mathfrak{g}_0$;

\medskip
\item  for the twisted symmetric
space examples $\mathcal{N}$  is
$\mathfrak{L}\backslash\mathfrak{N}$ where $\mathfrak{N}=
\mathfrak{g}_0+\mathfrak{e}_0$.

\end{itemize}
\begin{remark}\label{re:4} Notice that the neutral distribution for $\alpha_A$ coincides with the
homogeneous distribution into cosets of the centralizer of $A$, or its factor by $L$  in the case of   actions  on double coset spaces.
\end{remark}
\subsection{Exponential matrix coefficients decay on homogeneous space} Here we review some relevant facts concerning the preservation of spectral gaps when restricting to the subgroups in the example we consider.
\begin{theorem} \label{th:8} Let $S=S_1\times\cdots\times S_k$ be a product of noncompact simple Lie groups with finite center,  $\Gamma$ an irreducible lattice in $S$, and let $\rho_0$ stand for
the regular representation of $S$ on the subspace of $L^2(S/\Gamma)$ orthogonal to
constant functions. Then the restriction of $\rho_0$ to any simple factor of $S$ is isolated (in the Fell topology)
from the trivial representation.
\end{theorem}
If $k=1$ and the rank of $S$ is at least $2$ then $S$ has property $T$ and
the result follows directly from \cite{cow}. If $k=1$ and  the rank of $S$ is $1$, the spectral gap is already known. If $k\geq2$, in the case of nonuniform $\Gamma$, this was proved by Kleinbock and Margulis and appeared as Theorem 1.12 in \cite{Kleinbock1}.
L. Clozel \cite{Clozel} extended this result to congruence lattices discussed in \cite{Clozel}. By
Margulis \cite{Margulis}, $\Gamma$ is arithmetic and hence is commensurable
with a congruence lattice of the type in \cite{Clozel}, the result holds for all lattices.

We assume notations in Section \ref{sec:17} for the regular representation $(\pi,\mathcal{H})$ of $G$ or $G_\rho$, where $\mathcal{H}=L^2(G/\Gamma)$ or $L^2(G_\rho/\Gamma_\rho)$ correspondingly. It is clear that given any $a\in G$ (resp. $a\in G_\rho$)
\begin{align*}
  (\pi(a)f)(x)=f(a^{-1}x),\qquad \forall x\in X,\,\,f\in \mathcal{H}.
\end{align*}
For simplicity, we denote $\pi(a)f$ by $f(a^{-1})$ or by $f\circ a^{-1}$. Set $\mathcal{H}_0$ to be vectors in $\mathcal{H}$ orthogonal to constants. The following result is essential for later part:
\begin{corollary}\label{cor:1}
There exist constants $\gamma,E>0$, dependent only on $G$, $\Gamma$ and $\rho$ such that
if $h,\,f\in \mathcal{H}^1$, $i=1,2$ orthogonal to the constants in $\mathcal{H}$
then for any $a\in G$:
\begin{align}\label{for:16}
\abs{\langle h(a),f\rangle}&\leq
E(\norm{h}\norm{f}+\norm{h}_{1}\norm{f}\notag\\
&+\norm{f}_{1}\norm{h}+\norm{h}_{1}\norm{f}_{1})e^{-\gamma
\emph{dist}(e,a)}
\end{align}
where $\langle \cdot,\cdot\rangle$ denotes the inner product in $\mathcal{H}$
with respect to the Haar measure.
\end{corollary}
\begin{proof}
Let $G'$ denote the direct product of simple factors of $G$ and let $p$ be the projection of $G'$ to $G$. Then $G'/p^{-1}(\Gamma)$ is isomorphic to $G/\Gamma$. Since
$G$ has finite center then $p^{-1}(\Gamma)$ is also an irreducible lattice. This implies that \eqref{for:16} follows directly from Theorem \ref{th:8} and Theorem \ref{th:2} for vectors in $\mathcal{H}=L^2(G/\Gamma)$ orthogonal to constants.

For $\mathcal{H}=L^2(G_\rho/\Gamma_\rho)$, we note that $G_\rho/\Gamma_\rho$ can be viewed as a torus bundle over $G/\Gamma$. Set $V=\{f(g,t)\in \mathcal{H}_0: f(g,t)=\int_{T^N}f(g,t)dt\}$.  Then it is clear that $V$ is the set of $\RR^N$-invariant vectors in $\mathcal{H}_0$. For any $u\in \mathcal{H}_0$, write
\begin{align}\label{for:12}
    u(g,t)&=\underbrace{u(g,t)-\int_{T^N}u(g,t)dt}_{u^o(g,t)}+\underbrace{\int_{T^N}u(g,t)dt}_{u^1(g)}.
\end{align}
Then $f^0\in V^{\bot}$ and $f^1\in V$. Note that both $V^{\bot}$ and $V$ are closed and invariant under $G_\rho$. Hence we get a direct decomposition of $\mathcal{H}_0$ invariant under $G\ltimes\RR^N$. Then we have
\begin{align*}
\langle h(a),f\rangle=\langle h^o(a),f^o\rangle+\langle h^1(a),f^1\rangle.
\end{align*}
It follows from Theorem $1.2$ of \cite{Zhenqi23} that the restriction of $\pi$ to any simple factor of $G$ is isolated from the trivial representation. Then Theorem \ref{th:2} implies that
\begin{align}\label{for:7}
\abs{\langle h^o(a),f^o\rangle}&\leq
E_1(\norm{h^o}\norm{f^o}+\norm{h^o}_{1}\norm{f^o}\notag\\
&+\norm{f^o}_{1}\norm{h^o}+\norm{h^o}_{1}\norm{f^o}_{1})e^{-\gamma_1
\emph{dist}(e,a)}\notag\\
&\leq
E_1(\norm{h}\norm{f}+\norm{h}_{1}\norm{f}\notag\\
&+\norm{f}_{1}\norm{h}+\norm{h}_{1}\norm{f}_{1})e^{-\gamma_1
\emph{dist}(e,a)},
\end{align}
where $\gamma_1,\,E_1>0$, dependent only on $G$ and $\bar{X}$.

Note that $f^1,\,h^1$ can be viewed as functions in $L^2(G/\Gamma)$ orthogonal to constants. The arguments at the beginning of the proof show that
\begin{align}\label{for:11}
\abs{\langle h^1(a),f^1\rangle}&\leq
E_1(\norm{h^1}\norm{f^1}+\norm{h^1}_{1}\norm{f^1}\notag\\
&+\norm{f^1}_{1}\norm{h^1}+\norm{h^1}_{1}\norm{f^1}_{1})e^{-\gamma_2
\emph{dist}(e,a)}\notag\\
&\leq
E_2(\norm{h}\norm{f}+\norm{h}_{1}\norm{f}\notag\\
&+\norm{f}_{1}\norm{h}+\norm{h}_{1}\norm{f}_{1})e^{-\gamma_2
\emph{dist}(e,a)},
\end{align}
where $\gamma_2,\,E_2>0$, dependent only on $G$ and $X$.

Hence \eqref{for:16} follows from \eqref{for:7} and \eqref{for:11} immediately.
\end{proof}
\begin{remark}
Corollary \ref{cor:1} shows that for the restricted regular representation $(\pi,\mathcal{H}_0)$ of $G$ or $G_\rho$, its restriction to any
simple factor of $G$ is isolated from the trivial representation.
\end{remark}
\subsection{$\mathcal{H}^s$ and $\mathcal{H}_0^s$ space on $X$}
Fix a basis $\{Y_i\}_{1\leq i\leq q}$ of $\mathfrak L\backslash\mathfrak{g}$ or $\mathfrak L\backslash\mathfrak{g}_\rho$. Let
$\mathcal{H}^m_X$ to be the subspace of $L^2(X)$ such that $f$ and $Y_i^j (f)$, $1\leq j\leq m$ exist as $L^2$ functions for  $1\leq i\leq q$. Define
\begin{align*}
\norm{f}_m\stackrel{\text{def}}=(\sum_{i=1}^{q}\sum_{j=1}^m\norm{Y_i^jf}_0^2+\norm{f}_0^2)^{1/2}.
\end{align*}
Let
$\mathcal{H}_{0,X}^r\stackrel{\text{def}}{=}\{f\in\mathcal{H}^r_X\mid\int_{X}f=0\}$. We use subscripts to emphasize that the spaces $\mathcal{H}^m_X$ and $\mathcal{H}_{0,X}^r$ are different from the Hilbert spaces $\mathcal{H}^m$ and $\mathcal{H}_{0}^r$ for the regular representation $\pi$.

\subsection{Twisted cohomological stability for homogeneous space}\label{sec:4}
For a map $\mathcal{F}$ with coordinate functions $f_i$, $1\leq i\leq n_0$ and $-\infty\leq s\leq\infty$, we write $\mathcal{F}\in \mathcal{H}_X^s$ if $f_i\in \mathcal{H}_X^s$,
$1\leq i\leq n_0$; and define $\norm{\mathcal{F}}_s=\max_{1\leq i\leq
n_0}\norm{f_i}_s$. $\mathcal{F}\in\mathcal{H}_{0,X}^s$ is defined similarly.
For two maps $\mathcal{F}$, $\mathcal{G}$ define
$\norm{\mathcal{F},\mathcal{G}}_s=\max\{\norm{\mathcal{F}}_s,\norm{\mathcal{G}}_s\}$.
Write $\int_X \mathcal{F}=(\int_X
f_1d\mu,\cdots,\int_X f_{n_0}d\mu)$ where $\mu$ is the Haar measure.

In this part, we show the solvability condition  for the existence of a
solution to equation \eqref{eq:2}. This argument is essentially the  reduction of the vector values equation \eqref{eq:2} and solvability  condition in \eqref{for:38} to scalar equations. After showing the vanishing
of obstructions, the tame estimates for the solution of equation \eqref{eq:2} follow from
Theorem \ref{th:6}.

We use $\GG$ to denote $G$ or $G_\rho$. For any partially hyperbolic element $z\in \GG$, we use $\mathfrak{N}_z$ to denote the neutral distribution of $\text{Ad}(z)$.
\begin{lemma}\label{le:16}
For any partially hyperbolic element $z\in \GG$:
\begin{enumerate}
  \item \label{for:34} If $\mathcal{F}:X\to \mathfrak{N}_z$ and $\mathcal{F}\in \mathcal{H}_{0,X}^1$, then
  \begin{align*}
       \Lambda_{\binom{+}{-}}=\binom{-}{+}\sum_{\binom{j\geq0}{j\leq
-1}}\Ad(z)^{-(j+1)}(\mathcal{F}\circ z^j)
      \end{align*}
  are distributions.

  \medskip
 \item \label{for:38} There exist constants $m_1>\sigma_1>1$ (only depending on $\GG$) such that for any $l\geq m_1$, if $\mathcal{F}:X\to \mathfrak{N}_z$,
 $\mathcal{F}\in\mathcal{H}_{0,X}^l$ and
 \begin{align*}
   \sum_{j=-\infty}^{+\infty}\Ad(z)^{-(j+1)}(\mathcal{F}\circ
z^j)=0
 \end{align*}
 as a
distribution, then the equation
\begin{align}\label{eq:2}
\Lambda\circ z-\Ad(z)\Lambda=\mathcal{F}
\end{align}
have a solution $\Lambda\in\mathcal{H}_{0,X}^{l-\sigma_1}$ with estimates
\begin{align*}
\norm{\Lambda}_{l-\sigma_1}\leq C_{l,z} \norm{\mathcal{F}}_{l}.
\end{align*}

  \medskip
  \item \label{for:18} If $\mathcal{F}:X\to \mathfrak{N}_z$ and the equation $\Lambda\circ z-\Ad(z)\Lambda=\mathcal{F}$ has a solution $\Lambda\in\mathcal{H}_{0,X}^{1}$, then $\Lambda$ is unique.

  \medskip
  \item \label{for:36} There exist constants $m_1>\sigma_1>1$ (only depending on $\GG$) such that for any $l\geq m_1$, there exists $\delta(l)>0$, such that for any $b\in \GG$ with $d(z,b)<\delta$, if $\mathcal{F}:X\to \mathfrak{N}_z$ with $\Ad(b)$ neutral and invariant on $\mathfrak{N}_z$ and $\mathcal{F}\in\mathcal{H}_{0,X}^l$ satisfying
      \begin{align*}
        \sum_{j=-\infty}^{+\infty}\Ad(b)^{-(j+1)}(\mathcal{F}\circ
z^j)=0
      \end{align*}
as a
distribution, then the equation
\begin{align}\label{for:17}
 \Lambda\circ b-\Ad(b)\Lambda=\mathcal{F}
\end{align}
has a solution $\Lambda\in\mathcal{H}_{0,X}^{l-\sigma_1}$ and the following
estimate holds for $\Lambda$:
\begin{align*}
\norm{\Lambda}_{l-\sigma_1}\leq C_{l,z} \norm{\mathcal{F}}_{l}.
\end{align*}
\end{enumerate}
\end{lemma}
\begin{proof} Lift $\mathcal{F}$ from $X$ to $\bar{X}$, which we denote by $\mathcal{\tilde{F}}$. Then $\mathcal{\tilde{F}}\in\mathcal{H}_0^l$ if $\mathcal{F}\in \mathcal{H}_{0,X}^l$ for any $l$.

\noindent\eqref{for:34}: It is clear that if we can show $\Lambda_{\binom{+}{-}}$ are distributions for $\mathcal{\tilde{F}}$, then $\Lambda_{\binom{+}{-}}$ are also distributions for $\mathcal{F}$. Let $(a_n^{i,j})$ denote the matrix of $\textrm{Ad}(z)^n$ on $\mathfrak N_z$. Since the eigenvalues of $\textrm{Ad}(z)\mid_{\mathfrak N_z}$ are all in $U(1)$, we have
\begin{align}\label{for:20}
 \norm{(a_n^{i,j})}\leq C_z(\abs{n}+1)^{\dim\mathfrak N_z},\qquad \forall n\in\ZZ.
\end{align}
Denote by  $\Lambda_-^i$ and
$\Lambda_+^i$ the $i$-th coordinates of $\Lambda_-$ and $\Lambda_+$ respectively. We have
 \begin{align*}
       \Lambda^i_{\binom{+}{-}}=\binom{-}{+}\sum_{\binom{j\geq0}{j\leq
-1}}\sum_{k=1}^{\dim\mathfrak N_z}a_{-(j+1)}^{i,k} f_k\circ z^j
      \end{align*}
Then the conclusion follows directly from \eqref{for:60} of Lemma \ref{le:8}.

\medskip
\noindent \eqref{for:38}: If we can show the solution of lifted equation \ref{eq:2} on $\bar{X}$ is left-$L$ invariant with appropriate estimates (which is obvious from the expression of $\Lambda_{\binom{+}{-}}$ since $\mathcal{\tilde{F}}$ is left $L$ invariant ) then it descends to a map in $\mathcal{H}_{0,X}$ on $X$, which implies equation \ref{eq:2} has a solution on $X$ with desired estimates.

Let $\mathfrak N_{\mathbb C}$ be   the complexification  of
the subalgebra $\mathfrak N_z$.  There exists a basis in $\mathfrak N_{\mathbb C}$ such that
  in the basis $\textrm{Ad}(z)\mid_{\mathfrak N_z}$ has its Jordan normal form. As usual, this basis may be chosen  to consists
of  several real vectors and several pairs of complex conjugate
vectors. Let $J=(q^{i,j})$ be an $m\times m$ matrix which consists of blocks of $\textrm{Ad}(z)\mid_{\mathfrak N_z}$ corresponding to
the eigenvalue $\lambda$; i.e., let $q^{i,i}=\lambda$ for all $1\leq i\leq m$ and $q^{i,i+1}=0$ for all $1\leq i\leq m-1$ or $q^{i,i+1}=1$ for all $1\leq i\leq m-1$.

For any such block $J$ the equation \eqref{eq:2} has the form:
\begin{align}\label{eq:3}
&\Lambda\circ z-J\Lambda=\Theta
\end{align}
and  solvability condition  splits as
\begin{align}\label{for:161}
\sum_{j=-\infty}^{+\infty}J^{-(j+1)}\Theta\circ z^j=0
\end{align}
in this block.

$(1)$ \emph{The semisimple case}. Assume that $J$ is diagonalizable, i.e., $q^{i,i+1}=0$ for all $1\leq i\leq m-1$. Then equations \eqref{eq:3} and condition \eqref{for:161} split into
finitely many equations of the form
\begin{align}\label{for:175}
 \omega\circ z-\lambda\omega=\varphi
\end{align}
and
\begin{align}\label{eq:7}
\sum_{j=-\infty}^{+\infty}\lambda^{-(j+1)}\varphi\circ z^j=0
\end{align}
where $\varphi$ is a $\mathcal{H}_0^l$ function and $\lambda\in U(1)$ is the corresponding eigenvalue of $\textrm{Ad}(z)\mid_{\mathfrak N_z}$.
Then the conclusion follows directly from Theorem \ref{th:6}.

\medskip
$(2)$ \emph{The non-semisimple case}. Assume that $q^{i,i+1}=1$ for all $1\leq i\leq m-1$. We will show that the formal solutions
\begin{align}\label{for:172}
\Lambda_{\binom{+}{-}}=\binom{-}{+}\sum_{\binom{j\geq0}{j\leq
-1}}J^{-(j+1)}\Theta\circ z^j
\end{align}
are  in fact $\mathcal{H}_0^{l-2r-2}$ solutions.  Let the coordinate functions of $\Theta$ be $\vartheta_i$, $1\leq
i\leq m$. The $m$-th equation of \eqref{eq:3} becomes:
\begin{align}\label{for:32}
\omega_m\circ z-\lambda\omega_m=\vartheta_m
\end{align}
and the condition \eqref{for:161} splits as
\begin{align*}
\sum_{j=-\infty}^{+\infty}\lambda^{-(j+1)}\vartheta_m\circ z^j=0.
\end{align*}
Then the existence of a solution follows from Theorem \ref{th:6}.  Moreover, the estimate:
\begin{align*}
\norm{\omega_m}_{l-\sigma}\leq C_{l}\norm{\vartheta_m}_l\leq
C_{l}\norm{\Theta}_m
\end{align*}
holds.

Now we proceed by induction. Fix $i$ between $1$ and $m-1$ and assume that  for
all $1\leq k\leq m-i$, we have obtained a solution $\omega_{i+k}$ with the appropriate estimate, i.e., for
every $1\leq k\leq m-i$ we have a $\mathcal{H}_0^{l-(m-i-k+1)\sigma}$ function $\omega_{k+i}$ which solves the $k+i$-th
equation:
\begin{align*}
\omega_{k+i}\circ z-\lambda\omega_{k+i}=\vartheta_{k+i}+\omega_{k+i+1}
\end{align*}
and that the following estimates
\begin{align*}
\norm{\omega_{k+i}}_{l-(m-i-k+1)\sigma}\leq C_{l}\norm{\vartheta_{k+i}+\omega_{k+i+1}}_{l-(m-i-k)\sigma}\leq
C_{l}\norm{\Theta}_l
\end{align*}
hold every $1\leq k\leq m-i$.

We wish to find $\omega_i$ that solves the $i$-th equation of \eqref{eq:3}:
\begin{align}\label{for:13}
\omega_{i}\circ z-\lambda\omega_{i}=\vartheta_{i}+\omega_{i+1}
\end{align}
providing $\Lambda^i_+-\Lambda^i_{-}$ is a $0$ distribution.

The $i$-th coordinate function of $J^{-(j+1)}\Theta\circ z^j$ is:
\begin{align*}
\lambda^{-(j+1)}\vartheta_i\circ z^j+\sum_{k=1}^{m-i}C_{j,k}\vartheta_{k+i}\circ z^j,
\end{align*}
where $C_{j,k}=\frac{(-1)^k}{k!}\lambda^{-(j+1+k)}(j+1)\cdots (j+k)$; and  it follows that
\begin{align*}
 \Lambda_-^i-\Lambda_+^i=\sum_{j=-\infty}^{+\infty}\lambda^{-(j+1)}\vartheta_i\circ z^j+\sum_{j=-\infty}^{+\infty}
 \sum_{k=1}^{m-i}C_{j,k}\vartheta_{k+i}\circ z^j.
\end{align*}
Note that both $\Lambda_-^i$ and $\Lambda_+^i$ are distributions by \eqref{for:60} of Lemma \ref{le:8}. By substituting $\vartheta_{k+i}=\omega_{k+i}\circ z-\lambda\omega_{k+i}-\omega_{k+i+1}$ for all $1\leq k\leq m-i$ into $\sum_{k=1}^{m-i}C_{j,k}\vartheta_{k+i}\circ z^j$ we get
\begin{align*}
\sum_{k=1}^{m-i}C_{j,k}\vartheta_{k+i}\circ z^j&=\sum_{k=1}^{m-i-1}C_{j,k}(\omega_{k+i}\circ z-\lambda\omega_{k+i}-\omega_{k+i+1})\circ z^j\\
&+C_{j,m-i}(\omega_m\circ z-\lambda\omega_m)\circ z^j.
\end{align*}
By noting that
\begin{align*}
\sum_{j=-\infty}^{+\infty}C_{j,k}(\omega_{k+i}\circ z-\lambda\omega_{k+i})\circ z^j= \sum_{j=-\infty}^{+\infty}C_{j,k-1}\omega_{k+i}\circ z^j
\end{align*}
it follows that:
\begin{align*}
\Lambda_-^i-\Lambda_+^i=\sum_{j=-\infty}^{+\infty}\lambda^{-(j+1)}\vartheta_i\circ z^j+\sum_{j=-\infty}^{+\infty}\lambda^{-(j+1)}\omega_{i+1}\circ z^j.
\end{align*}
By assumption \eqref{for:161} the left side is a $0$
distribution.  Thus the equation \eqref{for:13} satisfies the solvability
and use Theorem \ref{th:6} again  to conclude that there exists a $\mathcal{H}_0^{l-(m-i-1)\sigma}$ solution $\omega_i$ with estimate
\begin{align*}
\norm{\omega_{i}}_{l-(m-i-1)\sigma}\leq C_{l}\norm{\vartheta_{k+i}+\omega_{k+i+1}}_{l-(m-i-k)\sigma}\leq
C_{l}\norm{\Theta}_l
\end{align*}
Since $k$ is an arbitrary integer between $1$ and $m-1$ it follows that there exists
a solution $\Lambda$ to equation \ref{eq:3} providing that the condition \ref{for:161} is satisfied.  This can be
repeated for all corresponding blocks $\textrm{Ad}(z)\mid_{\mathfrak N_z}$. Since the maximal size of a
Jordan block is bounded by $\dim\mathfrak N_z$, we obtain the following estimates for the solution $\Lambda$:
\begin{align*}
\norm{\Lambda}_{l-\sigma_1}\leq C_{l} \norm{\mathcal{F}}_{l}.
\end{align*}

\medskip

\noindent \eqref{for:18}: It suffices to show that if the lifted equation $\tilde{\Lambda}\circ z-\Ad(z)\tilde{\Lambda}=\tilde{\mathcal{F}}$, if $\tilde{\mathcal{F}}=0$ then $\tilde{\Lambda}=0$. We assume notations in \eqref{for:34}. Since $\tilde{\Lambda}=\Ad(z^{-j})\tilde{\Lambda}(z^j)$ for any $j\in\ZZ$, we get
 \begin{align*}
       \Lambda_i=\sum_{k=1}^{\dim\mathfrak N_z}a_{-j}^{i,k} \Lambda_k\circ z^j,\qquad \forall j\in\ZZ,
      \end{align*}
 where $\Lambda_i$ is the $i$-th coordinates of $\tilde{\Lambda}$.

 By using \eqref{for:20} and estimates obtained in \eqref{for:64} of Lemma \ref{le:8}, for any $u\in \mathcal{H}^1_0$ we get
\begin{align*}
    \langle \Lambda_i,\,u\rangle=\lim_{j\rightarrow\infty}\langle \sum_{k=1}^{\dim\mathfrak N_z}a_{-j}^{i,k} \Lambda_k\circ z^j,\,u\rangle=0
\end{align*}
Then $\Lambda_i=0$ follows immediately for any $i$. This shows that $\Lambda=0$.

\medskip

\noindent \eqref{for:36}: By using arguments in the proof of \eqref{for:38}, we can assume that $\mathcal{F}$ is on $\bar{X}$. If $\textrm{Ad}(b)$ is invariant on $\mathfrak N_z$, there exists a basis in $\mathfrak N_{\mathbb C}$ such that under this basis $\textrm{Ad}(b)$ has Jordan block form and each element in the basis has length $1$. As usual, this basis may be chosen  to consists
of  several real vectors and several pairs of complex conjugate
vectors. Let $J=(q^{i,j})$ be an $m\times m$ matrix which consists of blocks of $\textrm{Ad}(b)$ corresponding to
the eigenvalue $\lambda$; i.e., let $q_1^{i,i}=\lambda$ for all $1\leq i\leq m$; and $q_1^{i,j}=0$ if $j\neq i+1$ for all $1\leq i\leq m-1$. Note that since elements in the basis are of length $1$, then $q^{i,i+1}$ is not necessarily $0$ or $1$. It is clear that the norms of all blocks of $\textrm{Ad}(b)\mid_{\mathfrak N_z}$ are uniformly bounded from above for any $b$ sufficiently close to $z$.

If $b$ is sufficiently close to $z$, $b$ is partially hyperbolic. It follows from \eqref{for:38} and \eqref{for:18} that the equation \eqref{for:17} has a unique solution $\Lambda\in\mathcal{H}^{l-\sigma_1}$. Let the coordinate functions of $\Theta$ and $\Lambda$ be $\vartheta_i$ and $\omega_i$ respectively, $1\leq
i\leq m$. The $m$-th equation of \eqref{eq:3} becomes:
\begin{align}\label{for:26}
\omega_m\circ b-\lambda\omega_m=\vartheta_m
\end{align}
and for $1\leq k\leq m-1$ the $k$-th equation is
\begin{align}\label{for:27}
\omega_{k}\circ b-\lambda\omega_{k}=\vartheta_{k}+q^{k,k+1}\omega_{k+1}
\end{align}
The estimates of $\omega_m$ follow directly from equation \eqref{for:26} and Theorem \ref{th:6}:
\begin{align*}
\norm{\omega_{m}}_{l-\sigma}\leq C_{l}\norm{\vartheta_{m}}_{l-\sigma}\leq
C_{l}\norm{\Theta}_l.
\end{align*}
Inductively, the estimates of $\omega_i$ follow directly from equations \eqref{for:26} and \eqref{for:27} for $k\leq i\leq m$ and Theorem \ref{th:6}:
\begin{align*}
\norm{\omega_{i}}_{l-(m-i-1)\sigma}\leq C_{l}\norm{\vartheta_{i}+q^{i,i+1}\omega_{i+1}}_{l-(m-i)\sigma}\leq
C_{l}\norm{\Theta}_l.
\end{align*}
Repeated the above process for all corresponding blocks $\textrm{Ad}(b)\mid_{\mathfrak N_z}$. We obtain the following estimates for the solution $\Lambda$:
\begin{align*}
\norm{\Lambda}_{l-\sigma_1}\leq C_{l} \norm{\mathcal{F}}_{l}.
\end{align*}

\end{proof}

\section{Proof of Theorem \ref{th:3}}

\subsection{Reduction to finding a solution for a single cocycle equation}

\begin{lemma}\label{le:11}
Suppose $\alpha_A$ is a higher rank partially hyperbolic action on $X$ and $\mathfrak{N}$ is the neutral distributions of $\alpha_A$ on $\bar{X}$. If $R$ is a $\mathcal{H}^1$ \text{Ad}-twisted cocycle over $\alpha_A$ valued on $\mathfrak{N}$ and $\Omega\in \mathcal{H}_{0,X}^1$ solves the equation
\begin{align}\label{eq:5}
    \Omega\circ z-\Ad(z)\Omega=R_z+c
\end{align}
for some $e\neq z\in A$ and $c\in\mathcal{N}$ (Here we use $R_z:=R(z,\cdot)$ to denote the map from $X$ to $\mathfrak{N}$).  Then $\Omega$ solves \eqref{eq:5} for all the elements of the cocycle
i.e. there exists a homomorphism $\mathfrak{c}:A\rightarrow \mathfrak{N}$ such that for all $d\in A$ we have
\begin{align*}
    \Omega\circ d-\Ad_{d}\Omega=R_d+\mathfrak{c}(d).
\end{align*}
\end{lemma}
\begin{proof} From \eqref{eq:5} we see that $c=-\int_{X}R_zd\mu$ (note that $\Omega\in \mathcal{H}_{0,X}^1$). Set $R_d'=R_d-\int_{X}R_d d\mu$ for any $d\in A$. Then $R'\in \mathcal{H}_{0,X}^1$. It is easy to check that $R'$ is also an \text{Ad}-twisted cocycle over $\alpha_A$. By the twisted cocycle condition
\begin{align*}
    R_d'\circ z-\textrm{Ad}(z)R_d'=R_z'\circ d-\textrm{Ad}(d)R_z
\end{align*}
we have
\begin{align*}
\mathcal{T}_{d}R_z'=\mathcal{T}_{z}R_d',\qquad \forall d\in A,
\end{align*}
where $\mathcal{T}_{d}f=f\circ d-\textrm{Ad}(d)f$ for any $f\in \mathcal{H}_X$.

Substituting \eqref{eq:5} into the above equation we have
\begin{align*}
\mathcal{T}_{d}(\mathcal{T}_{z}\Omega)=\mathcal{T}_{z}R_d',\qquad \forall d\in A.
\end{align*}
Since operators $\mathcal{T}_{d}$ and $\mathcal{T}_{z}$ commute this implies
\begin{align*}
\mathcal{T}_{z}(\mathcal{T}_{d}\Omega-R_d')=0.
\end{align*}
By \eqref{for:18} of Lemma \ref{le:16} it follows that $\mathcal{T}_{d}\Omega-R_d'$, i.e., $\Omega$ solves
\begin{align*}
 \Omega\circ d-\Ad(d)\Omega=R_d+\mathfrak{c}(d)
\end{align*}
as well.
\end{proof}
\begin{remark}
The result still holds if we change $R$ to a cocycle over $\alpha_A$, since we just need to show that the operator $\mathcal{D}_{d}f=f\circ d-f$ is injective, which is obvious from Fact \ref{fact:1}.
\end{remark}
The first part of Theorem \ref{th:3} follows directly from Theorem \ref{th:9} and the above remark. Next, we prove the second part. Lemma \ref{le:11} shows that obtaining a tame solution of cocycle $R$ for one regular
generator suffices for the proofs of Theorem \ref{th:3}.
Hence to prove Theorem \ref{th:3}, it is equivalent to prove the following lemma:
\begin{lemma}\label{le:1}
Suppose $m_1$ and $\sigma_1$ are as defined in Lemma \ref{le:16}. Also suppose $a_1$ and $b_1$ commute and are regular generators for $\alpha_A$ on $X$.
If $\mathcal{F}\,,\mathcal{G}:X\rightarrow
\mathfrak{N}$ satisfying
\begin{align}\label{for:28}
  \mathcal{F}\circ b_1-\emph{Ad}(b_1)\mathcal{F}=\mathcal{G}\circ a_1-\emph{Ad}(a_1)\mathcal{G}
\end{align}
and $\mathcal{F},\,\mathcal{G}\in\mathcal{H}_{0,X}^m$, $m\geq m_1$, then the equations
\begin{align}\label{eq:6}
&\Omega\circ a_1-\emph{Ad}(a_1)\Omega=\mathcal{F}\notag\\
&\Omega\circ b_1-\emph{Ad}(b_1)\Omega=\mathcal{G}
\end{align}
have a common  solution $\Omega\in\mathcal{H}_{0,X}^{m-\sigma_1}$ with the
following estimate
\begin{align}\label{for:49}
\norm{\Omega}_{m-\sigma_1}\leq C_{m}
\norm{\mathcal{F},\mathcal{G}}_{m}.
\end{align}

\end{lemma}
\begin{proof}
By arguments in the proof of \eqref{for:38} of Lemma \ref{le:16}, we can assume $\mathcal{F}\,,\mathcal{G}:\bar{X}\rightarrow
\mathfrak{N}$. Next, we show that the obstructions to solving equations \ref{eq:6} vanish. From the equation \eqref{for:28} we get
\begin{align*}
&\sum_{j=-n}^{j=n}\Ad(a_1)^{-(j+1)}\mathcal{F}\circ (b_1a_1^j)-\sum_{j=-n}^{j=n}\Ad(b_1)\Ad(a_1)^{-(j+1)}\mathcal{F}\circ a_1^j\\
&=\Ad(a_1)^{-(n+1)}\mathcal{G}\circ a_1^{n+1}-\Ad(a_1)^{n}\mathcal{G}\circ a_1^{-n}.
\end{align*}
From proof of \eqref{for:34} of Lemma \ref{le:16} and \eqref{for:60} of Lemma \ref{le:8} we see that  the right-hand converges to  $0$  as a distribution (in $\pi$) when
$n\rightarrow \infty$. Hence  we see that
\begin{align}\label{for:39}
\Ad(b_1)^{-1}\sum_{j=-\infty
}^{+\infty}\Ad(a_1)^{-(j+1)}\mathcal{F}\circ (b_1a_1^j)=\sum_{j=-\infty
}^{+\infty}\Ad(a_1)^{-(j+1)}\mathcal{F}\circ a_1^j
\end{align}
as distributions. This shows that
\begin{align*}
\Ad(b_1)^{-k}\sum_{j=-\infty
}^{+\infty}\Ad(a_1)^{-(j+1)}\mathcal{F}\circ (b_1^ka_1^j)=\sum_{j=-\infty
}^{+\infty}\Ad(a_1)^{-(j+1)}\mathcal{F}\circ a_1^j
\end{align*}
as distributions for any $k\in\ZZ$.

If we can show that
\begin{align}\label{for:29}
  \sum_{n=-\infty
}^{+\infty}\sum_{j=-\infty
}^{+\infty}\Ad(b_1)^{-(n+1)}\Ad(a_1)^{-(j+1)}\mathcal{F}\circ(b_1^na_1^j)
\end{align}
is a distribution, then by iterating equation \eqref{for:39}
we obtain:
\begin{align*}
&\sum_{n=-\infty
}^{+\infty}\sum_{j=-\infty
}^{+\infty}\Ad(b_1)^{-(n+1)}\Ad(a_1)^{-(j+1)}\mathcal{F}\circ(b_1^na_1^j)\\
&=\sum_{n=-\infty
}^{+\infty}\Ad(b_1)^{-1}\sum_{j=-\infty
}^{+\infty}\Ad(b_1)^{-n}\Ad(a_1)^{-(j+1)}\mathcal{F}\circ (b_1^na_1^j)\notag\\
&=
\sum_{n=-\infty
}^{+\infty}\Ad(b_1)^{-1}\sum_{j=-\infty
}^{+\infty}\Ad(a_1)^{-(j+1)}\mathcal{F}\circ a_1^j
\end{align*}
Since the series in the left hand side of above equation is a distribution by \eqref{for:34} of Lemma \ref{le:16}, it forces
$\sum_{j=-\infty
}^{+\infty}\Ad(a_1)^{-(j+1)}\mathcal{F}\circ a_1^j$ to be a $0$
distribution. Then the result follows from Lemma \ref{le:16}.

Let $(a_{n,j}^{k,l})$  denote the matrix of $\Ad(a_1)^{n}\Ad(b_1)^{j}$ on $\mathcal{N}$. Since the eigenvalues of $\Ad(a_1)$ and $\Ad(b_1)$ on $\mathcal{N}$ are all in $U(1)$, we have
\begin{align}\label{for:42}
 \norm{(a_{n,j}^{k,l})}\leq C(\abs{n}+1)^{\dim\mathcal{N}}(\abs{j}+1)^{\dim\mathcal{N}},\qquad \forall n,\,j\in\ZZ.
\end{align}
The $i$-th coordinate of \eqref{for:29} is:
\begin{align*}
 \sum_{j\in\ZZ} \sum_{i\in\ZZ}\sum_{k=1}^{\dim\mathcal{N}}a_{-(j+1),-(n+1)}^{i,k} \mathcal{F}_k\circ(b_1^na_1^j)
\end{align*}
where $\mathcal{F}_k$ is the $i$-th coordinate of $\mathcal{F}$.

By using \eqref{for:42} it follows from \eqref{for:40} of Lemma \ref{le:9} that \eqref{for:29} is a distribution.
\end{proof}

\end{document}